\documentclass{siamltex}
\usepackage{amsmath,amssymb,epsfig,url}

\newcommand{\Cc}{\mathbb{C}}
\newcommand{\colspan}{\mathrm{colspan}}
\newcommand{\dd}{\mathrm{d}}
\newcommand{\err}{\epsilon}
\newcommand{\et}{\tilde{\err}}
\newcommand{\geqs}{\geqslant}
\newcommand{\Hh}{\underline{H}}
\newcommand{\hhat}{\widehat{h}}
\newcommand{\Hhat}{\widehat{H}}
\newcommand{\Hht}{\underline{\tilde{H}}}
\newcommand{\Ht}{\tilde{H}}

\newcommand{\Kk}{\mathcal{K}}
\newcommand{\leqs}{\leqslant}
\newcommand{\Matlab}{\textsc{Matlab}}
\newcommand{\Pe}{\textit{P\!e}}
\newcommand{\Rr}{\mathbb{R}}

\newcommand{\Rrnn}{\Rr^{n\times n}}
\newcommand{\Span}{\mathrm{span}}
\newcommand{\uhat}{\widehat{u}}
\newcommand{\ut}{\tilde{u}}
\newcommand{\vhat}{\widehat{v}}
\newcommand{\Vhat}{\widehat{V}}

\title{Residual, restarting and Richardson iteration\\
       for the matrix exponential, revised}

\author{Mike A. Botchev\thanks{Department of Applied Mathematics,
University of Twente, P.O.~Box 217, 7500~AE Enschede, the Netherlands,
\texttt{mbotchev@na-net.ornl.gov}.}}
\date{\today}

\begin{document}
\maketitle

\centerline{\emph{To the memory of my father}}

\begin{abstract}
A well-known problem in computing some matrix functions iteratively 
is the lack of a clear, commonly accepted residual notion.
An important matrix function for which this is the case is the matrix 
exponential.  Suppose the matrix exponential of a given matrix times a given 
vector has to be computed.  We develop the approach of Druskin, Greenbaum and
Knizhnerman (1998) and interpret the sought-after vector as
the value of a vector function satisfying the linear system of ordinary
differential equations (ODE) whose coefficients form the given matrix.
The residual is then defined with respect to the initial-value problem 
for this ODE system.  The residual introduced in this way can be seen as 
a backward error.  We show how the residual can be computed efficiently 
within several iterative methods for the matrix exponential.
This completely resolves the question of reliable stopping criteria 
for these methods.  Further, we show that the residual concept
can be used to construct new residual-based iterative methods.
In particular, a variant of the Richardson method for the new residual 
appears to provide an efficient way to restart Krylov subspace methods
for evaluating the matrix exponential.
\end{abstract}



\section{Introduction}
Matrix functions, and particularly the matrix exponential, have been an 
important tool in scientific computations for decades (see e.g.\
\cite{Gantmacher,GolVanL,Higham_bookFM,FrommerSimoncini08a,AlmohyHigham2010}).  
The lack of a clear notion for a residual for many matrix functions
has been a known problem in iterative computation of matrix functions
\cite{BenziRazouk07,FrommerSimoncini08a,EshofHochbruck06}.
Although it is possible to define a residual for some matrix functions such as 
the inverse or the square root, for many important matrix functions including 
the matrix exponential, sine and cosine, no natural notion for residuals
seems to exist.  


Assume for given $A\in\Rrnn$, such that $A+A^*$ is 
positive semidefinite, and $v\in\Rr^n$ the vector 
\begin{equation}
\label{expm0}
y=\exp(-A)v
\end{equation}
has to be computed.  The question is how to evaluate the quality of 
an approximate solution 
\begin{equation}
\label{expm}
y_k\approx\exp(-A)v,
\end{equation}
where $k$ refers to the number of steps (iterations) needed to construct $y_k$.
We interpret the vector $y$ as the value of a vector function $y(t)$ at $t=1$
such that
\begin{equation}\label{ivp}
y'(t) = -Ay(t),\quad y(0)=v.
\end{equation}
The exact solution of this initial-value problem (IVP) is given by
\begin{equation*}
\label{expmt}
y(t)=\exp(-tA)v.
\end{equation*}
Assuming now that there is a vector function $y_k(t)$ such that $y_k(1)=y_k$,
we define the residual for $y_k(t)\approx y(t)$ as
\begin{equation}
\label{resid}
r_k(t) \equiv -Ay_k(t)-y_k'(t).
\end{equation}
The key point in this residual concept is that $y=\exp(-A)v$
is seen not as a problem on its own but rather as the exact solution
formula for the problem~\eqref{ivp}.  The latter provides the equation
where the approximate solution is substituted to yield the residual.
We illustrate this in Table~\ref{t1}, where the introduced matrix exponential
residual is
compared against the conventional residual for a linear system $Ay=v$.
As can be seen in the Table, the approximate solution
satisfies a perturbed IVP, where the perturbation is the residual.
Thus, the introduced residual can be seen as a backward error
(see Section~\ref{sect:err_est} for residual-based error estimates).
If one is interested in computing the matrix
exponential $\exp(-A)$ itself, then the residual can be defined
with respect to the matrix IVP
$$
X'(t)=-A X(t), \quad X(0) = I,
$$ 
with the exact solution $X(t)=\exp(-tA)$.
Checking the norm of $r_k(t)$ in~\eqref{resid} is proposed as a possible 
stopping criterion of Krylov subspace iterations 
first in~\cite{DruskinGreenbaumKnizhnerman98} and more recently
for a similar matrix function in~\cite{KnizhnermanSimoncini09}.

\begin{table}
\caption{The linear system and matrix exponential residuals.  In both
cases the sought-after vector is $f(A)v$, with either $f(x)=1/x$ 
or $f(x)=\exp(-x)$.  The error is defined as the exact solution minus
the approximate solution.}\label{t1}
\begin{center}
\begin{tabular}{cccl}\hline\hline
$f(x)$              & $1/x$     &~~~~  & \hspace*{4em}$\exp(-x)$          \\
\hline
exact solution $y$  & $y=A^{-1}v$       & &  $\begin{aligned}        
                                              &\text{define }y(t)=\exp(-tA)v,\\ 
                                              &\text{set }y:=y(1)
                                              \end{aligned}$   \\
\hline
residual equation  & $Ay=v$            & &$\left\{ \begin{aligned}
                                           y'(t)&= -Ay(t)\\      
                                           y(0) &= v            
                                           \end{aligned}\right.$ \\
\hline
residual for $y_k\approx y$&$r_k=v-Ay_k$& & $r_k(t)=-Ay_k(t)-y_k'(t)$\\ 
\hline
\begin{tabular}{c}
mapping\\
error $\err_k$ $\rightarrow$ residual
$r_k$\end{tabular}         &  $r_k=A\err_k$&  & $\left\{\begin{aligned}
                                             r_k(t) &= \err'_k(t)+A\err_k(t)\\
                                             \err_k(0)  &= 0
                                             \end{aligned}\right.$\\
\hline
\begin{tabular}{c}
perturbed problem\\
(backward stability)
\end{tabular}              &$Ay_k=v-r_k$& &$\left\{ \begin{aligned}
                                           y'_k(t)&= -Ay_k(t)-r_k(t)\\      
                                           y_k(0) &= v            
                                           \end{aligned}\right.$ \\
\hline
\end{tabular}
\end{center} 
\end{table}

The contribution of this paper is twofold.  
First, it turns out that the residual~\eqref{resid} can be efficiently computed 
within several iterative methods for matrix exponential evaluation.
We show how this can be done in several popular Krylov subspace 
and Chebyshev polynomial methods for computing $\exp(-A)v$.
Second, we show how the residual notion leads to new algorithms to compute the 
matrix exponential.  Two basic Richardson-like iterations are proposed and
discussed.  When combined with Krylov subspace methods, one of them can be seen 
as an efficient way to restart the Krylov subspace methods.
Furthermore, this approach for the matrix exponential residual can be readily
extended to the sine and cosine matrix functions (see the conclusion section).

The equivalence between problems~\eqref{expm} and~\eqref{ivp} has been widely used in 
numerical literature and computations.
In addition to already cited work~\cite{DruskinGreenbaumKnizhnerman98,KnizhnermanSimoncini09}
see e.g.\ the very first formula in \cite{19ways} or \cite[Section~10.1]{Higham_bookFM}.
Moreover, methods for solving~\eqref{expm} are applied to~\eqref{ivp} 
(for instance, exponential time integrators \cite{HochLubSel97,HochbruckOstermann2010}) and vice 
versa~\cite[Section~4]{19ways}.
In~\cite{EshofHochbruck06}, van den Eshof and Hochbruck represent
the error $\err_k(t)\equiv y(t)-y_k(t)$ as the solution of the IVP 
$\err_k'(t)=-A\err_k(t)+r_k(t)$, $\err_k(0)=0$
and obtain an explicit, non-computable expression for $\err_k(t)$.  This allows them
to justify a stopping criterion for their shift-and-invert Lanczos algorithm, based
on stagnation of the approximations.  
Although being used, especially in the field of numerical 
ODEs (see e.g.\ 
\cite{Enright2000,Shampine2005,Lubich2008book,KierzenkaShampine2008}), 
the exponential residual~\eqref{resid} does seem to have a potential
which has not been fully exploited yet, in particular in
matrix computations.  Our paper aims at filling this gap.

The paper is organized as follows.  Section~\ref{sect:Krylov} is 
devoted to the matrix exponential residual within Krylov subspace
methods.  In Section~\ref{sect:Chebyshev} we show how the Chebyshev
iterations can be modified to adopt the residual control.
Section~\ref{sect:err_est} presents some simple residual-based
error estimates.  Richardson iteration for the matrix exponential
is the topic of Section~\ref{sect:Richardson}.
Numerical experiments are discussed in
Section~\ref{sect:num_exp}, and conclusions are drawn in the 
last section.

Throughout the paper, unless noted otherwise, $\|\cdot\|$ denotes the Euclidean
vector 2-norm or the corresponding induced matrix norm.

\section{Matrix exponential residual in Krylov subspace methods}
\label{sect:Krylov}
Kry\-lov subspace methods have become an important tool for computing matrix
functions (see e.g.~\cite{Henk:f(A),DruskinKnizh89,Knizh91,GallSaad92,Saad92,DruskinKnizh95,%
HochLub97,DruskinKnizh98,HochLubSel97}).
For $A\in\Rrnn$ and $v\in\Rr^n$ given, the Arnoldi process yields, after $k$ steps, 
vectors $v_1$, \dots, $v_{k+1}\in\Rr^n$ that are orthonormal in exact arithmetic 
and span the Krylov subspace $\Kk_k(v,Av,\dots,A^{k-1}v)$ (see 
e.g.~\cite{GolVanL,SaadBook,Henk:book}).  
If $A=A^*$, the Lanczos process is usually 
used instead of Arnoldi.  Together with the basis vectors $v_j$, the Arnoldi or 
Lanczos processes deliver an upper-Hessenberg matrix 
$\Hh_k\in\Rr^{(k+1)\times k}$, 
such that the following relation 
holds~\cite{GolVanL,SaadBook,Henk:book}:
\begin{equation}
\label{Arnoldi}
\begin{aligned}
AV_k &= V_{k+1}\Hh_k,\quad\text{or equivalently,}\\
AV_k &= V_kH_k +   h_{k+1,k}v_{k+1}e_k^T,
\end{aligned}
\end{equation}
where $V_k\in\Rr^{n\times k}$ has columns $v_1$, \dots, $v_k$,
$H_k\in\Rr^{k\times k}$ is the matrix $\Hh_k$ without the last row 
$(0, \cdots, 0,  h_{k+1,k})$, and $e_k=(0, \cdots , 0 , 1)^T\in\Rr^k$. 
The first basis vector $v_1$ is the normalized vector $v$: $v_1=v/\|v\|$.

\subsection{Ritz-Galerkin approximation}
An approximation $y_k$ to the matrix exponential $y=\exp(-A)v$ 
is usually computed as $y_k(1)$, with
\begin{equation}
\label{Krylov1}
y_k(t) = V_k\exp(-t H_k)(\beta e_1),
\end{equation}
where $\beta=\|v\|$ and $e_1=(1,0, \dots, 0)^T\in\Rr^k$. 
An important property of the Krylov subspace is its scaling invariance: 
application of the Arnoldi process to $tA$, $t\in\Rr$, results in 
the upper-Hessenberg matrix 
$t\Hh_k$, and the basis vectors $v_1$, \dots, $v_{k+1}$ are independent of $t$.
It is convenient for us to write 
\begin{equation}
\tag{\ref{Krylov1}$'$}
y_k(t) = V_ku_k(t), \quad u_k(t)\equiv\exp(-t H_k)(\beta e_1),
\end{equation}
with $u_k(t):\Rr\rightarrow\Rr^k$ being the solution of the projected IVP
\begin{equation}
\label{proj_ivp}
u_k'(t)=-H_ku_k(t),\quad u_k(0)=\beta e_1.  
\end{equation}
The following simple Lemma (cf.\ \cite[formula~(29)]{DruskinGreenbaumKnizhnerman98})
provides an explicit  expression for the residual.

\begin{lemma}\label{lemma1}
Let $y_k(t)\approx y(t)=\exp(-tA)v$ be the Krylov subspace approximation given 
by~(\ref{Krylov1}).  Then for any $t\geqs 0$ the residual $r_k(t)$ 
for $y_k(t)\approx y(t)$ is 
$$
\begin{aligned}
  r_k(t)     &= -\beta h_{k+1,k} e_k^T\exp(-tH_k)e_1 v_{k+1}, \\
  \|r_k(t)\| &= |\beta h_{k+1,k} e_k^T\exp(-tH_k)e_1|=|h_{k+1,k}[u_k(t)]_k|,
\end{aligned}
$$
where $[u_k(t)]_k$ is the last entry of the vector function $u_k(t)$ defined
in~(\ref{Krylov1}$'$).
\end{lemma}
\begin{proof}
It follows from~\eqref{Krylov1} that $y_k'(t)=-V_kH_k\exp(-t H_k)(\beta e_1)$.
From the Arnoldi relation~\eqref{Arnoldi} we have 
$$
Ay_k(t)=AV_k\exp(-t H_k)(\beta e_1) 
= (V_kH_k + h_{k+1,k}v_{k+1}e_k^T)\exp(-t H_k)(\beta e_1),
$$
which yields the result:
$$
r_k(t)=-Ay_k(t)-y_k'(t)= -h_{k+1,k}v_{k+1}e_k^T\exp(-t H_k)(\beta e_1).
$$
\end{proof}

Note that the Krylov subspace approximation~\eqref{Krylov1} satisfies the initial 
condition $y_k(0)=v$ by construction:
$$
y_k(0) = V_k(\beta e_1)= \beta v_1 = v.
$$
Thus, there is no danger that the residual $r_k(t)=-Ay_k(t)-y_k'(t)$ is small in norm
for some $y_k(t)$ approaching a solution of the ODE system $y'=Ay$ with other
initial data.

The residual notion~\eqref{resid} allows us to see~\eqref{Krylov1}
as the Ritz-Galerkin approximation: the residual vector $r_k(t)$ is orthogonal,
for any $t\geqs 0$, to the search space $\Span(v_1,\dots,v_k)$:
\begin{equation}
\label{Galerkin}
V_k^*r_k(t)=V_k^*(-Ay_k(t)-y_k'(t))=V_k^*(-AV_ku_k(t)-V_ku_k'(t))=
-H_ku_k(t)-u_k'(t)=0.
\end{equation}
Here we used the relation $V_k^*AV_k=H_k$, which follows from~\eqref{Arnoldi}
if $V_k$ is orthonormal (this may not always be the case in 
floating point arithmetic).

The residual $r_k(t)$ turns out to be closely related to the 
so-called \emph{generalized residual} $\rho_k(t)$~\cite{HochLubSel97}.
Following~\cite{HochLubSel97} (see also~\cite{Saad92}), we can write
\begin{align*}
y_k(t)&=\beta V_k \exp(-tH_k)e_1 &&= 
\frac1{2\pi i}\oint_{\Gamma}e^\lambda V_k(\lambda I+tH_k)^{-1}\beta e_1\dd\lambda,
\\
y(t)  &=\exp(-tA)v &&= 
\frac1{2\pi i}\oint_{\Gamma}e^\lambda (\lambda I+tA)^{-1}v\dd\lambda,
\end{align*}
where $\Gamma$ is a closed contour in $\Cc$ encircling the spectrum of $A$.
Thus, $y_k(t)$ is an approximation to $y(t)$ where the resolvent inverse
$(\lambda I+tA)^{-1}v$ is approximated by $k$ steps of the fully orthogonal
method (FOM):
$$
\err_k = y(t)-y_k(t) = \frac1{2\pi i}\oint_{\Gamma}e^\lambda \mathrm{error}_k^{\mathrm{FOM}}
\dd\lambda.
$$
Since the FOM error is unknown, the authors of~\cite{HochLubSel97} 
replace it by the known FOM residual, which is 
$\beta (-th_{k+1,k})v_{k+1}e_k^T(\lambda I+tH_k)^{-1}e_1$.
This leads to the generalized residual
\begin{equation}
\label{gen_resid}
\begin{aligned}
\rho_k(t) &\equiv\frac{1}{2\pi i}
\oint_\Gamma e^\lambda\beta (-th_{k+1,k})v_{k+1}e_k^T(\lambda I+tH_k)^{-1}e_1\dd\lambda\\
&= -\beta t h_{k+1,k}e_k^T\exp(-t H_k)e_1 \, v_{k+1},
\end{aligned}
\end{equation}
which coincides, up to a factor $t$, with our matrix exponential residual $r_k(t)$.
For the generalized residual, this provides a justification which is
otherwise lacking: strictly speaking, there is no reason why the error 
in the integral expression above can be replaced by the residual.
In Section~\ref{sect:Krylov_tests}, a numerical test is presented
to compare stopping criteria based on $r_k(t)$ and $\rho_k(t)$.

\subsection{Shift-and-invert Arnoldi/Lanczos approximations}
In the shift-and-invert (SaI) Arnoldi/Lanczos 
approximations~\cite{MoretNovati04,EshofHochbruck06} the Krylov subspace is
built up with respect to the matrix $(I+\gamma A)^{-1}$, with $\gamma>0$
being a parameter, so that the Krylov basis matrix $V_{k+1}\in\Rr^{n\times(k+1)}$ and 
an upper-Hessenberg matrix $\Hht_k\in\Rr^{(k+1)\times k}$ are built such that 
(cf.~\eqref{Arnoldi})
\begin{equation}
\label{Arnoldi_sai}
\begin{aligned}
(I+\gamma A)^{-1}V_k &= V_{k+1}\Hht_k,\quad\text{or, equivalently,}\\
(I+\gamma A)^{-1}V_k &= V_k\Ht_k + \tilde{h}_{k+1,k}v_{k+1}e_k^T,
\end{aligned}
\end{equation}
where $\Ht_k\in\Rr^{k\times k}$ is the first $k$ rows of $\Hht_k$.
The approximation $y_k(t)\approx\exp(-tA)v$ is then computed as given by~\eqref{Krylov1},
with $H_k$ defined as~\cite{EshofHochbruck06}
\begin{equation}
\label{Arnoldi_saiH}
H_k = \frac1\gamma(\Ht_k^{-1}-I).
\end{equation}
Relation~\eqref{Arnoldi_sai} can be rewritten as 
(cf.\ formula~(4.1) in \cite{EshofHochbruck06})
\begin{equation}
\label{Arnoldi_sai2}
AV_k = V_kH_k - \frac{\tilde{h}_{k+1,k}}{\gamma}(I+\gamma A)v_{k+1}e_k^T\Ht_k^{-1},
\end{equation}
which leads to the following lemma.

\begin{lemma}\label{lemma2}
Let $y_k(t)\approx y(t)=\exp(-tA)v$ be the SaI Krylov subspace 
approximation~(\ref{Krylov1}),
with $H_k$ defined in (\ref{Arnoldi_saiH}).  Then for any $t\geqs 0$
the residual $r_k(t)$ for $y_k(t)\approx y(t)$ is 
$$
\begin{aligned}
  r_k(t)     &= \beta\,\frac{\tilde{h}_{k+1,k}}{\gamma}\, e_k^T\Ht_k^{-1}\exp(-tH_k)e_1 (I+\gamma A)v_{k+1}, \\
  \|r_k(t)\| &\leqs \beta\left|\frac{\tilde{h}_{k+1,k}}{\gamma}\right| 
             | e_k^T\Ht_k^{-1}\exp(-tH_k)e_1| (1+\gamma\|A\|).
\end{aligned}
$$
\end{lemma}
\begin{proof}
The proof is very similar to that of Lemma~\ref{lemma1}.
Instead of the conventional Arnoldi relation~\eqref{Arnoldi}, 
relation~\eqref{Arnoldi_sai2} should be used.
\end{proof}

\subsection{Error estimation in Krylov subspace methods}
If $y_k(t)$ is a Krylov subspace approximation to $y(t)=\exp(-tA)v$,
the error function $\err_k(t)\equiv y(t)-y_k(t)$ satisfies the IVP
\begin{equation}
\label{err1}
\err_k'(t)  =-A\err_k(t) + r_k(t),\quad \err_k(0)= 0.
\end{equation}
To estimate the error, this equation can be solved approximately by any 
suitable time integration scheme; for example, by Krylov exponential schemes
as discussed e.g.\ in \cite[Section~4]{GallSaad92} or \cite{HochLubSel97}.
The time integration process for solving~\eqref{err1} can further be 
optimized to take into account that the residual function $r_k(t)$ depends
on time as $r_k(t)=\psi_k(t)v_{k+1}$ with $v_{k+1}=\const$ 
and $\psi_k(t)$ being a scalar function of $t$ (see Lemma~\ref{lemma1}):
$$
\psi_k(t)\equiv -\beta h_{k+1,k} e_k^T\exp(-tH_k)e_1.
$$ 
Van den Eshof and Hochbruck~\cite{EshofHochbruck06} propose to get an error 
estimate by replacing in $\err_k(t)\equiv y(t)-y_k(t)$ the exact solution 
$y(t)$ with the same continued Krylov process approximation $y_{k+m}(t)$:
\begin{equation}
\begin{aligned}
\err_k(t) &\approx y_{k+m}(t)-y_k(t) = V_{k+m}u_{k+m}(t)-V_ku_k(t)=V_{k+m}\et_k(t),
\\
\|\err_k(t)\| &\approx \|\et_k(t)\|=\|u_{k+m}(t)-\ut_k(t)\|,  
\end{aligned}  
\end{equation}
where 
$$
V_ku_k(t)=V_{k+m}\ut_k(t),\quad
\ut_k(t) = [\, (u_k(t))^T,\underbrace{0,\dots, 0}_{m} \,]^T
$$
and $u_k(t)$ and $u_{k+m}(t)$ are the solutions of the projected 
IVP~\eqref{proj_ivp} obtained with respectively $k$ and $k+m$ 
Krylov steps.
It is not difficult to see that in this case $\et_k(t)\equiv u_{k+m}(t)-\ut_k(t)$
is the Galerkin solution of~\eqref{err1} with respect to the subspace $\colspan V_{k+m}$.
Indeed, we have 
\begin{align*}
y'_{k+m} &=-Ay_{k+m}-r_{k+m}(t),  \quad & y_{k+m}(t) &= V_{k+m}u_{k+m}(t),
\\
y'_{k  } &=-Ay_{k  }-r_{k  }(t),  \quad & y_{k  }(t) &= V_{k+m}\ut_k  (t).
\end{align*}
Subtracting $y_k'$ from $y_{k+m}'$ and multiplying the result from the left 
by $V_{k+m}^*$ (in assumption the orthonormality of $V_{k+m}$ is not spoiled
in floating point arithmetic) we obtain
$$
(u_{k+m}(t)-\ut_k(t))'=-H_{k+m}(u_{k+m}(t)-\ut_k(t)) + V_{k+m}^*r_k(t),
\quad V_{k+m}^*r_k(t)=\psi_k(t)e_{k+1},
$$
and we arrive at the projected IVP
\begin{equation}
\label{err2}
\et_k'(t)=-H_{k+m}\et_k(t) + \psi_k(t)e_{k+1},
\end{equation}
where $e_{k+1}$ is the $(k+1)$th basis vector in $\Rr^{k+m}$.
This shows that error estimation by the same continued Krylov process
is a better option than solving the correction equation~\eqref{err1}
by a new Krylov process: the latter would mean that we neglect
the built up subspace.  In fact, solving IVP~\eqref{err1} by another
process and then correcting the approximate solution $y_k(t)$ 
can be seen as a restarting of the Krylov subspace.  We explore this 
approach further in Section~\ref{sect:Richardson}.

\section{Matrix exponential residual for Chebyshev approximations}
\label{sect:Chebyshev}
A well-known method to compute $y_m(t)\approx\exp(-tA)v$ is
based on the Chebyshev polynomial expansion (see for instance
\cite{TalEzer89,RyabenkiiTsynkov}):
\begin{equation}
\label{chebyshev_expansion}
y_m(t)=P_m(-tA)v = \left[\sum_{k=1}^mc_kT_k(-tA) + \frac{c_0}2I\right]v.
\end{equation}
Here we assume that the matrix $tA$ can be transformed to have its
eigenvalues within the interval $[-1,1]\subset\Rr$
(for example, $A$ can be a Hermitian or a skew-Hermitian matrix).  
Here, $T_k$ are the Chebyshev polynomials of the 
first kind, whose actions on the given vector $v$ can be computed 
by the Chebyshev recursion
\begin{equation}
\label{chebyshev_recursion}
T_0(x)=1,\quad
T_1(x)=x,\quad
T_{k+1}(x) = 2xT_k(x)-T_{k-1}(x),\quad k=1,2,\dots,
\end{equation}
and the coefficients $c_k$ can be computed, for a large $M$, as
\begin{equation}
\label{c_k}
c_k=\frac2M\sum_{j=1}^M\exp(\cos(\theta_j))\cos(k\theta_j),\quad
k=0,1,\dots,m,\quad \theta_j=\frac{\pi(j-\frac12)}M,
\end{equation}
which means interpolating $\exp(x)$ at the Chebyshev polynomial roots
(see e.g.~\cite[Section~3.2.3]{RyabenkiiTsynkov}).
This Chebyshev polynomial approximation is used for evaluating
different matrix functions in~\cite{BenziRazouk07}.

The recursive algorithm \eqref{chebyshev_expansion}--\eqref{c_k}
can be modified to provide, along with $y_m(t)$, vectors $y_m'(t)$ and $Ay_m(t)$,
so that the exponential residual $r_m(t)\equiv-Ay_m(t)-y_m'(t)$ can be
controlled in the course of the iterations.  To do this, we use the well-known
relations 
\begin{align}
T_k'(x)&=kU_{k-1}(x),
\label{Tkprime}
\\
xT_k(x)&=\frac12(T_{k+1}(x)+T_{k-1}(x)),
\label{xTk}
\\
xU_k(x)&=\frac12(U_{k+1}(x)+U_{k-1}(x)),
\label{xUk}
\\
T_k(x)&=\frac12(U_{k}(x)-U_{k-2}(x)),
\label{Uk2Tk}
\end{align}
where $k =1,2,\dots$ and $U_k$ are the Chebyshev polynomials of the second kind:
\begin{equation}
\label{chebyshev_recursion2}
U_0(x)=1,\quad
U_1(x)=2x,\quad
U_{k+1}(x) = 2xU_k(x)-U_{k-1}(x),\quad k=1,2,\dots.
\end{equation}
For~\eqref{Uk2Tk} to hold for $k=1$ we denote $U_{-1}(x)=0$.
From~\eqref{chebyshev_expansion},\eqref{Tkprime} and~\eqref{xUk} it follows that
\begin{equation}
\label{w_prime_chebyshev}
\begin{aligned}
y_m'(t)&= \left[\sum_{k=1}^m\frac{c_k}t(-tA)T_k'(-tA)\right]v\\
       &= \left[\sum_{k=1}^m\frac{c_kk}{2t}(U_k(-tA)+U_{k-2}(-tA))\right]v,
\quad m=1,2,\dots.
\end{aligned}
\end{equation}
Similarly, from~\eqref{chebyshev_expansion}, \eqref{xTk} and \eqref{Uk2Tk}, we 
obtain
\begin{equation}
\label{minAw}
\begin{aligned}
-Ay_m(t)
&=\left[\sum_{k=1}^m\frac{c_k}{2t}(T_{k+1}(-tA)+T_{k-1}(-tA)) -\frac{c_0}2A\right]v
\\
&=\left[\sum_{k=1}^m\frac{c_k}{2t}(U_{k+1}(-tA)-U_{k-3}(-tA))-\frac{c_0}2A\right]v,
\quad m=1,2,\dots,
\end{aligned}
\end{equation}
where we define $U_{-2}(x)=-1$.

The obtained recursions can be used to formulate an algorithm for 
computing $y_m(t)\approx\exp(-tA)v$ that controls the residual $r_m(t)=-Ay_m(t)-y_m'(t)$, 
see Figure~\ref{fig:chebyshev_alg}.
Just like the original Chebyshev recursion algorithm for the matrix exponential,
it requires one action of the matrix $A$ per iteration.  To be able to control
the residual, more vectors have to be stored than in the conventional
algorithm: 8 instead of 4.

\begin{figure}
$$
\begin{aligned}
& u_{-2} := -v, \quad u_{-1} := 0, \quad u_0 := v, \quad u_1 := (-2*t)*(A*v)
\\
& \text{compute }c_0
\\
& y:= (0.5*c_0)*u_0,\quad y':= 0, \quad \mathtt{minusAy}  := (c_0/t/4)*u_1
\end{aligned}
$$
\vspace*{-4ex}$$
\begin{aligned}
\text{for }& k=1,\dots,N_{\max}\\
           & u_2 := (-2*t)*(A*u_1) - u_0\\
           & \text{compute }c_k\\
           & y  := y  + (c_k/2)    *(u_1-u_{-1})\\
           & y' := y' + (c_k*k/2/t)*(u_1+u_{-1})\\
           & \mathtt{minusAy}  := \mathtt{minusAy}  + (c_k/4/t)  *(u_2-u_{-2})\\
           & u_{-2} := u_{-1}\\
           & u_1 := u_0\\
           & u_0 := u_1\\
           & u_1 := u_2\\
           & \mathtt{resnorm} := \|\mathtt{minusAy} - y'\|\\
           &\text{if }\mathtt{resnorm} < \mathtt{toler}\\
           &\quad \text{break}\\
           &\text{end}\\
\text{end }&
\end{aligned}
$$ 
\caption{Chebyshev expansion algorithm to compute the vector $y_{N_{\max}}(t)\approx\exp(-tA)v$.
The input parameters are $A\in\Rrnn$, $v\in\Rr^n$, $t>0$ and $\mathtt{toler}>0$.  
It is assumed
that the eigenvalues $\lambda$ of $tA$ satisfy $-1\leqslant\lambda\leqslant 1$.}
\label{fig:chebyshev_alg}
\end{figure}

\section{Residual-based error estimates}
\label{sect:err_est}
By definition of the residual~\eqref{resid}, the approximate 
solution $y_k(t)\approx\exp(-tA)v$ is the exact solution of the problem
\begin{equation}
\label{ivp_pert}
y_k'(t)=-Ay_k(t)-r_k(t),\quad y(0)=v,
\end{equation}
which is a perturbation of the original problem~\eqref{ivp}.
Therefore the residual $r_k(t)$ can be seen as the backward error for $y_k(t)$.
From \eqref{ivp_pert} and~\eqref{ivp} it is easy to see that the error $\err_k(t)$
satisfies the initial-value problem
\begin{equation}
\label{ivp_error}
\err_k'(t)= -A\err_k(t) + r_k(t),\quad \err_k(0) = 0,
\end{equation}
with the exact solution 
\begin{equation}
\label{e_formula}
\err_k(t)=\int_0^t\exp((s-t)A)r_k(s)\dd s. 
\end{equation}
This formula can be used to obtain error bounds in terms
of the norms of the matrix exponential and the residual~\cite{Acta_stab}.

\begin{lemma}\label{lemma_e_estim}
Let $|A|$ denote a matrix whose entries are absolute values
of the entries of $A$.  Let $r_k(t):\Rr\rightarrow\Rr^n$ be continuous 
for $t\geqs 0$ and $\bar{r}_k(t)$ be a vector-function with entries 
$$
[\bar{r}_k(t)]_i=\max_{s\in[0,t]} |[r_k(s)]_i|, \quad i=1,\dots,n.
$$
It holds for any $t\geqs 0$ that
\begin{equation}
\label{e_estim}
\|\err_k(t)\|_*\leqs \| |t\varphi(-t A)|\,\bar{r}_k(t)\|_*
\leqs \| |t\varphi(-t A)|\|_*\|\bar{r}_k(t)\|_*,
\end{equation}
with $\|\cdot\|_*$ being any consistent matrix (vector) norm
and $\varphi(x)=(\exp(x)-1)/x$.
\\
Note that, for any $B\in\Rrnn$,
$\| |\varphi(B)|\|_*=\|\varphi(B)\|_*$ for
the 1- and $\infty$- norms.
\end{lemma}

\begin{proof}
For simplicity of notation, throughout the proof we omit 
the subindex $\cdot_k$ and write $\err_k(t)=\err(t)$
and $r_k(t)=r(t)$.  Denote, for a fixed $t$, the entry $(i,j)$ 
of the matrix $\exp((s-t)A)$ by $e_{ij}(s)$.
Entry $i$ of $\err(t)$ can be bounded as
\begin{equation}
\label{lemma_rel1}
\begin{aligned}
\left|[\err(t)]_i\right|
&=\left|\int_0^t\sum_{j=1}^ne_{ij}(s)[r(s)]_j\dd s \right|
 =\left|\sum_{j=1}^n\int_0^te_{ij}(s)[r(s)]_j\dd s \right|
\\
&\leqs\sum_{j=1}^n\left|\int_0^te_{ij}(s)[r(s)]_j\dd s \right|.    
\end{aligned}
\end{equation}
For any fixed $t\geqs 0$, $e_{ij}(s)$ is an infinitely differentiable
function in $s$, represented by the uniformly convergent power series
$$
e_{ij}(s)=\left[ I + (s-t)A + \dots + \frac{(s-t)^k}{k!}A^k + \dots \right]_{ij}.
$$
Therefore, on the interval $[0,t]$ the function $e_{ij}(s)$ changes 
its sign a finite number of times\footnote{This follows from the fact that 
a smooth function can have only a finite number of isolated roots
on a bounded interval.}.
Denote the points where $e_{ij}(s)$ changes its sign in $[0,t]$ 
by $t_1$, $t_2$, \dots, $t_{m-1}$, with $m$ depending on the indices $i$ and
$j$ of $e_{ij}(s)$.  Setting $t_0=0$ and $t_m=t$, we obtain
$$
\int_0^te_{ij}(s)[r(s)]_j\dd s = \sum_{l=1}^{m}
\int_{t_{l-1}}^{t_l}e_{ij}(s)[r(s)]_j\dd s,
$$
where $e_{ij}(s)$ is either nonnegative or nonpositive on each of the
subintervals $[t_{l-1},t_l]$.  Since the function $[r(s)]_j$ is continuous, 
a version of the mean-value theorem for the integral
\cite[Theorem~5, Section~6.2.3]{ZorichI} can be applied to each of the integrals
under the summation:
$$
\int_0^te_{ij}(s)[r(s)]_j\dd s = \sum_{l=1}^{m}
[r(\xi_l)]_j\int_{t_{l-1}}^{t_l}e_{ij}(s)\dd s,
$$
with $0=t_0\leqs\xi_1\leqs t_1\leqs\xi_2\leqs t_2\leqs\dots\leqs t_{m-1}\leqs\xi_m\leqs t_m=t$.
This yields the bounds
\begin{gather*}
\min_{0\leqs s\leqs t}[r(s)]_j\sum_{l=1}^{m}\int_{t_{l-1}}^{t_l}e_{ij}(s)\dd s
\leqs \int_0^te_{ij}(s)[r(s)]_j\dd s \leqs
\max_{0\leqs s\leqs t}[r(s)]_j\sum_{l=1}^{m}\int_{t_{l-1}}^{t_l}e_{ij}(s)\dd s,
\\
\min_{0\leqs s\leqs t}[r(s)]_j\int_{0}^{t}e_{ij}(s)\dd s
\leqs \int_0^te_{ij}(s)[r(s)]_j\dd s \leqs
\max_{0\leqs s\leqs t}[r(s)]_j\int_{0}^{t}e_{ij}(s)\dd s,
\\
\left|\int_0^te_{ij}(s)[r(s)]_j\dd s\right| \leqs
\underbrace{\max\left\{\left|\min_{0\leqs s\leqs t}[r(s)]_j\right|,
                \left|\max_{0\leqs s\leqs t}[r(s)]_j\right|\right\}
           }_{\displaystyle=\max_{0\leqs s\leqs t}\left|[r(s)]_j\right|=[\bar{r}(t)]_j}
\left|\int_0^te_{ij}(s)\dd s\right|.
\end{gather*}
Evaluating
$$
\int_0^t e_{ij}(s)\dd s = 
\int_0^t \left[ I + (s-t)A + \dots + \frac{(s-t)^k}{k!}A^k + \dots \right]_{ij}\dd s=
[t\varphi(-tA)]_{ij}
$$
and substituting
$$
\left|\int_0^te_{ij}(s)[r(s)]_j\dd s\right| 
\leqs\left|[t\varphi(-tA)]_{ij}\right| \, [\bar{r}(t)]_j
$$
into \eqref{lemma_rel1} yields
$$
\left|[\err(t)]_i\right|\leqs
\sum_{j=1}^n \left|[t\varphi(-tA)]_{ij}\right| \, [\bar{r}(t)]_j.
$$
\end{proof}


The estimates provided by the last lemma can be specified further
if more information on $A$ is available.  For symmetric
positive definite $A$ in the 2-norm holds
$$
\|\err_k(t)\| \leqs t\|\bar{r}_k(t)\|, \quad t\geqs 0.
$$
This estimate appeared in \cite[formula~(32)]{DruskinGreenbaumKnizhnerman98}.
If the eigenvalues $\lambda_i$ of $|A|$ are known to lie 
in the interval $[\lambda_{\min},\lambda_{\max}]$, with $\lambda_{\min}>0$,
then
$$
\| |t\varphi(-t A)|\| = \| t\varphi(-t \Lambda)\| =| t\varphi(-t \lambda_{\min})|,
$$
where $\Lambda$ is a diagonal matrix with the eigenvalues of $|A|$ as its entries.

\section{Richardson iteration for the matrix exponential}
\label{sect:Richardson}
The notion of the residual allows us to introduce a Richardson
method for the matrix exponential.
\subsection{Preconditioned Richardson iteration}
Consider the preconditioned Ri\-chard\-son iterative method 
\begin{equation}
\label{rich_ls}
x_{k+1} = x_k + M^{-1}r_k  
\end{equation}
for solving a linear system $Ax=b$, with the preconditioner $M\approx A$
and residual $r_k=b-Ax_k$.
Note that $M^{-1}r_k$ is an approximation to the unknown error $A^{-1}r_k$.
By analogy with~\eqref{rich_ls}, one can formulate the Richardson method
for the matrix exponential as 
\begin{equation}
\label{rich}
y_{k+1}(t)=y_k(t) + \et_k(t),  
\end{equation}
where $\et_k\approx\err_k$ is the approximate solution of the IVP~\eqref{err1}.
One option, which we follow here, is to choose a suitable 
$M\approx A$ and let $\et_k$ be the solution of the IVP
\begin{equation}
\label{err_rich}
\et_k'(t)  =-M\et_k(t) + r_k(t),\quad \et_k(0)= 0.
\end{equation}
Just as for solving linear systems, $M$ has to compromise
between the approximation quality $M\approx A$ and the ease of 
solving~\eqref{err_rich}.

In fact, the exponential Richardson method can be seen as a relative
of the waveform relaxation methods for solving ODEs, see
e.g.~\cite{LumsdaineWu1997,JanssenVandewalle1997}.
The key difference of method~\eqref{rich}--\eqref{err_rich} from 
the waveform relaxation methods is that the latter are merely time-stepping
methods.  In particular, in waveform relaxation methods
relation~\eqref{err_rich} is not solved as such but replaced by a discretization, 
e.g.\ by a linear multistep integration formula. 

The residual $r_k(t)$ of the Richardson iteration~\eqref{rich}--\eqref{err_rich} 
can be shown to satisfy the following  
recursion.  From~\eqref{rich} and~\eqref{err_rich} we have 
$$
-y'_{k+1}(t) = -y'_k(t) + M\et_k(t)-r_k(t).
$$
Subtracting relation $Ay_{k+1}(t)=Ay_k(t)+A\et_k(t)$ from this equation, 
we get
\begin{equation}
\label{res_rich}
r_{k+1}(t)=-y'_k(t)+M\et_k(t) - r_k(t)-Ay_k(t) - A\et_k(t)=
(M-A)\et_k(t).
\end{equation}
Taking into account that 
\begin{equation*}
\label{et_formula}
\et_k(t)=\int_0^t\exp((s-t)M)r_k(s)\dd s,
\end{equation*}
we obtain (cf.~\cite{LumsdaineWu1997})
\begin{equation}
\label{res_rich1}
r_{k+1}(t)=(M-A)\et_k(t)=(M-A)\int_0^t\exp((s-t)M)r_k(s)\dd s.
\end{equation}
Using relation~\eqref{e_estim}, we arrive at the following result.

\begin{lemma}
Let $|A|$ and $\bar{r}_k(t)$ be as defined in Lemma~\ref{lemma_e_estim}.
The residual $r_k(t)=-y'_k(t)-Ay_k(t)$ in the exponential Richardson
method~\eqref{rich} satisfies for any $t\geqs 0$
\begin{align*}
\|r_{k+1}(t)\|_* &\leqs \| |t(M-A)\varphi(-t M)|\,\bar{r}_k(t)\|_*  \\
                &\leqs \| |t(M-A)\varphi(-t M)|\|_*\|\bar{r}_k(t)\|_*,\\
\text{so that}\quad
\max_{s\in[0,t]}\|r_{k+1}(s)\|_* &\leqs
\max_{s\in[0,t]}\| |s(M-A)\varphi(-s M)|\|_*
\max_{s\in[0,t]}\|r_k(s)\|_*,
\end{align*}
with $\|\cdot\|_*$ being any consistent matrix (vector) norm
and $\varphi(x)=(\exp(x)-1)/x$.
\end{lemma}

\begin{proof}
The proof follows the lines of the proof of 
Lemma~\ref{lemma_e_estim}.
\end{proof}

The estimate provided by the lemma shows that, at least for some matrices
$A$ and $M$ and not too large $t\geqs 0$, the exponential 
Richardson iteration converges faster than the Richardson iteration 
for linear system solution.
Indeed, since
$$
t(M-A)\varphi(-tM) = (M-A)M^{-1}(I-\exp(-tM)),
$$
the upper bounds for the residual reduction are
\begin{align*}
\text{linear system Richardson:}&\quad  
\|r_{k+1}\|_*\leqs \| (M-A)M^{-1}\|_*\|r_k\|_*,
\\
\text{exponential Richardson:}&\quad  
\|r_{k+1}(t)\|_*\leqs \| | (M-A)M^{-1}(I-\exp(-tM)) | \|_*\|\bar{r}_k(t)\|_*, 
\end{align*}
with $t\geqs 0$ in the second inequality (both sides of the inequality are
zero for $t=0$).
For general matrices $A$ and $M$ it is hard to prove that 
$$
\| | (M-A)M^{-1}(I-\exp(-tM)) | \| \leqs \| (M-A)M^{-1}\|, \qquad t\geqs 0.
$$
This inequality holds in the 2-norm, for instance, if $A$ is an $M$-matrix
and $M$ is its diagonal part (in this case the matrices $M-A$, $M^{-1}$ 
and $I-\exp(-tM)$ are elementwise nonnegative and we can get rid of 
the absolute value sign).  As can be seen in Figure~\ref{fig:Richardson_bounds},
exponential Richardson can converge reasonably well even when 
$\| (M-A)M^{-1}\|$ is hopelessly close to one.

\begin{figure}
\centering\epsfig{file=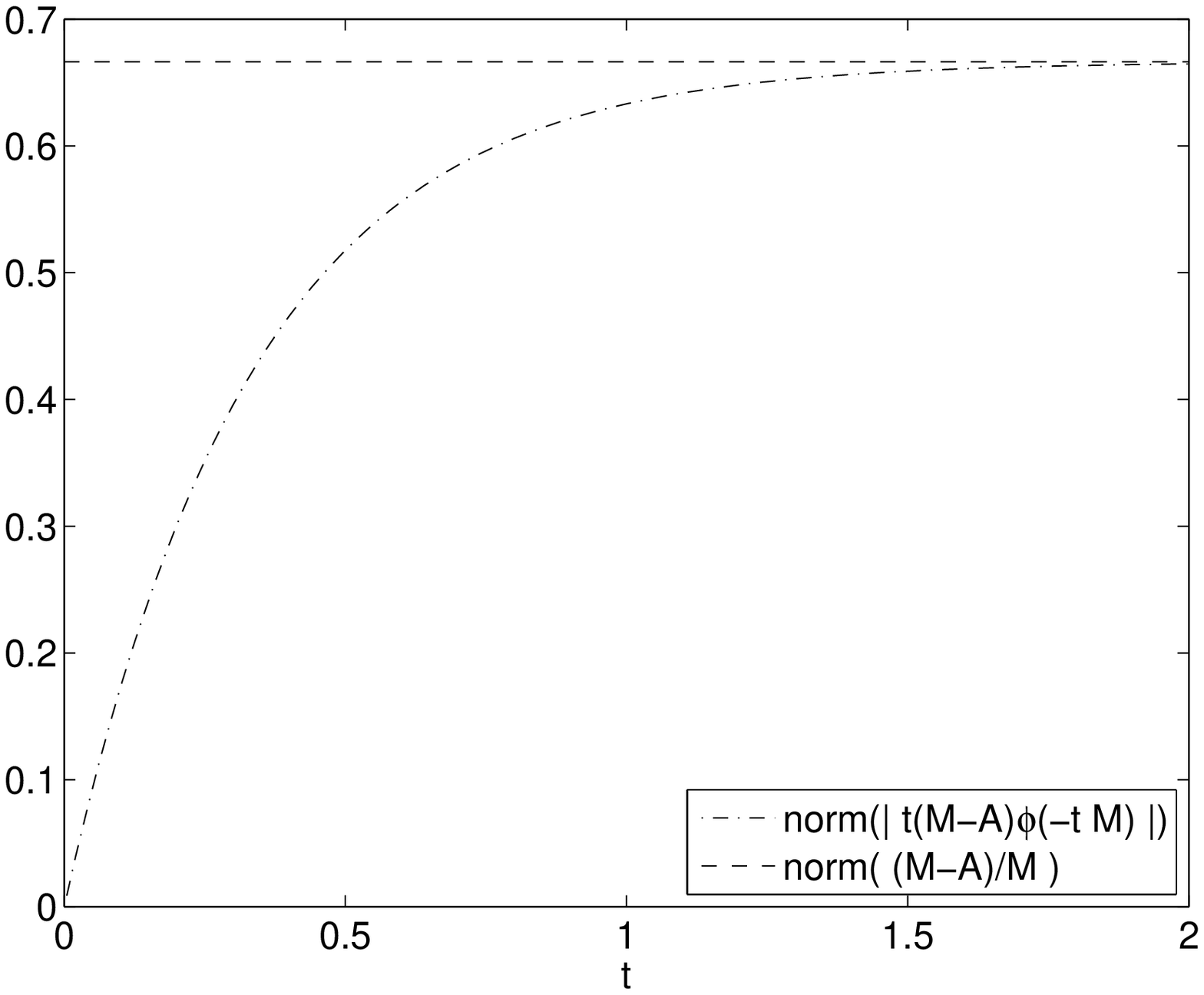, width=0.6\linewidth}\\
\centering\epsfig{file=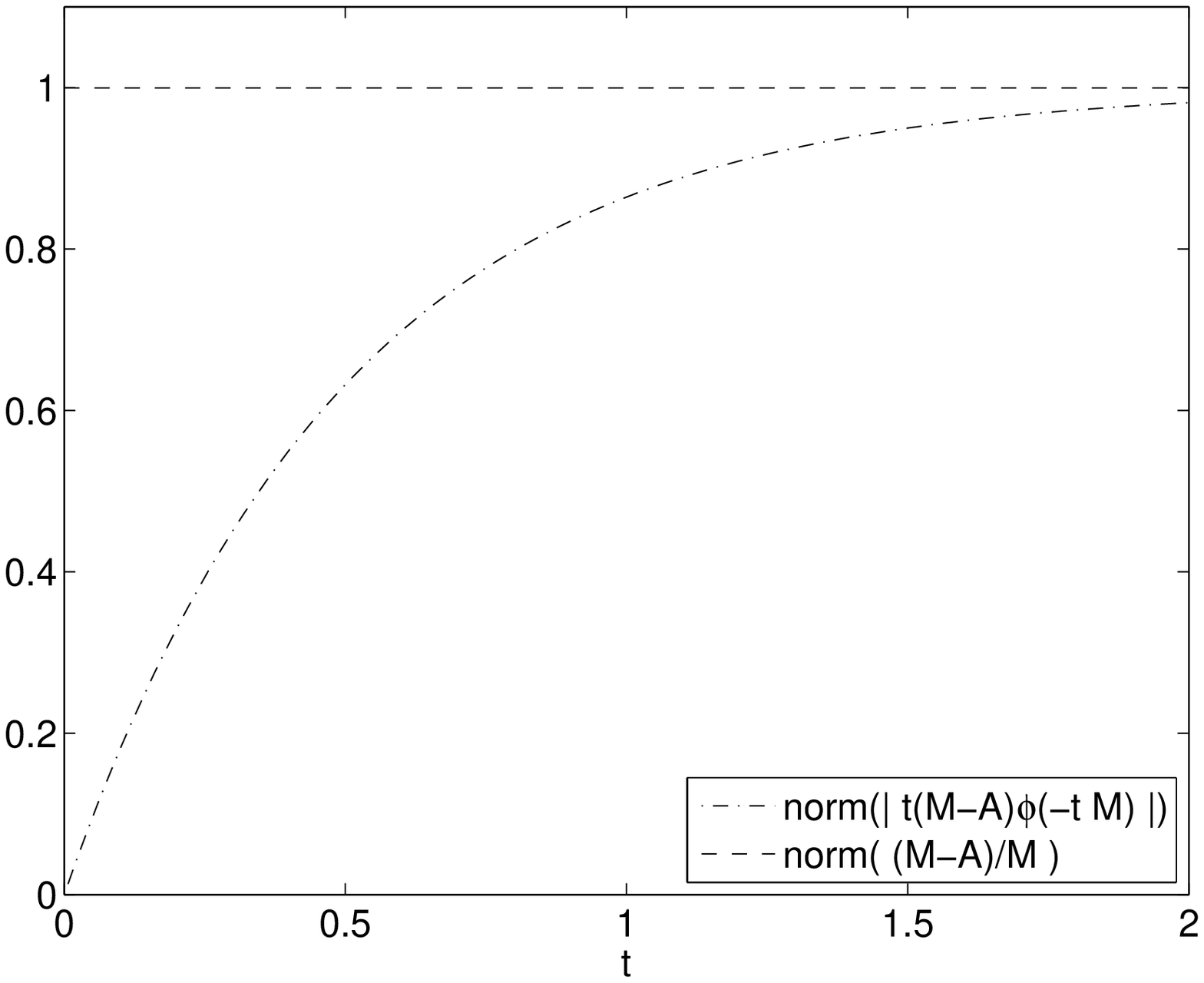,width=0.6\linewidth}
\caption{Upper bounds for residual reduction in Richardson iteration 
for the linear system (dashed) and matrix exponential (dash-dotted).
$A=\mathrm{tridiag} (-1, 3, -1)$ (top) and 
$A=\mathrm{tridiag} (-1, 2, -1)$ (bottom).  In both cases 
$A\in\Rr^{100\times 100}$, $M=\mathrm{diag}(A)$.}
\label{fig:Richardson_bounds}  
\end{figure}

An important practical issue hindering the use of the exponential Richardson
iteration is the necessity to store the vectors $r_k(t)$ for different $t$.
To achieve a good accuracy,
sufficiently many samples of $r_k(t)$ have to be stored.  Our limited experience
indicates that the exponential Richardson iteration can be of interest
if the accuracy requirements are relatively low, say up to $10^{-5}$.
In the experiments described in Section~\ref{sect:numerRich} just 20~samples 
were sufficient to get the residual below tolerance $10^{-4}$
for a matrix of size $n=10^4$.

\subsection{Krylov restarting via Richardson iteration}
\label{sect:Krylov-Richardson}
In the exponential Ri\-chard\-son iteration~\eqref{rich} the error $\et_k(t)$
does not have to satisfy~\eqref{err_rich}, which is just one possible 
choice for $\et_k(t)$.  Another choice is to take $\et_k(t)$ to be the 
Krylov approximate solution of the IVP
\begin{equation}
\label{ivp1}
\et_k'=-A\et_k+r_k(t),\quad \et_k(0)=0.
\end{equation}
If the approximate solution $y_k(t)$ is also obtained by a
Krylov process, then the Richardson iteration~\eqref{rich}, \eqref{ivp1}
can be seen as a restarted Arnoldi/Lanczos method for computing $\exp(-tA)v$.
Indeed, assume that the IVP~\eqref{ivp1} is solved approximately by $m$ 
Arnoldi or Lanczos steps, so that the next Richardson approximation is
\begin{equation}
\label{rich_rst}
y_{k+m}(t)=y_k(t) + \et_k(t).  
\end{equation}
Assume $y_k(t)$ is the Krylov or SaI Krylov approximation
to $\exp(-tA)v$, given by~\eqref{Krylov1}, \eqref{Arnoldi} 
or by~\eqref{Krylov1}, \eqref{Arnoldi_saiH} respectively.
To derive an expression for $r_{k+m}(t)$, we first notice that
\begin{equation}
\label{res_rst0}
r_k(t) = \psi_k(t) w_k,  \quad
\psi_k: \Rr\rightarrow \Rr, \quad w_k=\const\in\Rr^n,
\end{equation}
with a scalar function $\psi_k(t)$ and a constant vector $w_k$.
These are given by
$$
\psi_k(t)=-\beta h_{k+1,k} e_k^T\exp(-tH_k)e_1,\quad
w_k=v_{k+1}
$$
for the regular Krylov approximation (see Lemma~\ref{lemma1}) and by
$$
\psi_k(t)=\beta\,\frac{\tilde{h}_{k+1,k}}{\gamma}\, 
e_k^T\Ht_k^{-1}\exp(-tH_k)e_1 (I+\gamma A)v_{k+1},
\quad
w_k= (I+\gamma A)v_{k+1}
$$
for the shift-and-invert Krylov approximation (see Lemma~\ref{lemma2}).
The error 
$$
\err_k(t)=y(t)-y_k(t)=\int_0^t\exp((s-t)A)r_k(s)\dd s=
\int_0^t\psi_k(s)\exp((s-t)A)w_k\dd s
$$ 
is approximated by the $m$-step Krylov solution~$\et_k(t)$ of 
\eqref{ivp1}:
\begin{equation}
\label{et_k}
\begin{aligned}
\et_k(t) &=
\int_0^t\psi_k(s) \Vhat_m\exp((s-t)\Hhat_m) \|w_k\|e_1\dd s\\
&=\Vhat_m\underbrace{\int_0^t \exp((s-t)\Hhat_m)\psi_k(s) \|w_k\|e_1 \dd s}_%
{\uhat(t)},  
\end{aligned}
\end{equation}
where $e_1=(1,0,\dots,0)^T\in\Rr^m$ and $\Vhat_m\in\Rr^{n\times m}$
and $\Hhat_m\in\Rr^{m\times m}$ result from $m$ steps of the Arnoldi/Lanczos
process for the matrix $A$ and the vector $w_k$.
It is not difficult to see that $\uhat(t)$ is the solution of the IVP
\begin{equation}
\label{ivp_uhat}
\uhat'(t)=-\Hhat_m\uhat(t)+\psi_k(t)\|w_k\|e_1,\quad \uhat(0)=0.
\end{equation}
From~\eqref{et_k} and~\eqref{ivp_uhat}, we have 
\begin{equation}
\label{res_rst2}
\begin{aligned}
&r_{k+m}(t) 
 =  -y'_{k+m}(t) -Ay_{k+m}(t) = -y'_k(t)-\et'_k(t)-Ay_k(t)-A\et_k(t)\\
&= r_k(t)-\Vhat_m\uhat'(t)-A\Vhat_m\uhat(t) =
   r_k(t)-\Vhat_m(-\Hhat_m\uhat(t)+\psi_k(t)\|w_k\|e_1)-A\Vhat_m\uhat(t)\\
&= r_k(t) +\Vhat_m\Hhat_m\uhat(t)
   -\underbrace{\psi_k(t)\|w_k\|\Vhat_m e_1}_{r_k(t)}-A\Vhat_m\uhat(t) =
   (\Vhat_m\Hhat_m-A\Vhat_m)\uhat(t).
\end{aligned}
\end{equation}
If $\Vhat$ and $\Hhat$ result from the conventional Arnoldi/Lanczos
process, then (cf.\ \eqref{Arnoldi}) 
$\Vhat_m\Hhat_m-A\Vhat_m=-\hhat_{m+1,m}\vhat_{m+1}e_m^T$, so that
\begin{equation}
\label{res_rst2a}
r_{k+m}(t) = -\hhat_{m+1,m}[\uhat(t)]_m\vhat_{m+1},
\end{equation}
with $[\uhat(t)]_m$ being the last component of $\uhat(t)$.
If $\Vhat$ and $\Hhat$ are obtained with the SaI Ar\-nol\-di/Lanc\-zos
process then (cf.\ \eqref{Arnoldi_sai2})
$$
\Vhat_m\Hhat_m-A\Vhat_m = {\widehat{\tilde{h}}_{m+1,m}}\gamma^{-1}
(I+\gamma A)\vhat_{m+1}e_m^T\widehat{\tilde{H}}_m^{-1},
$$
with all quantities defined 
by~\eqref{Arnoldi_sai}--\eqref{Arnoldi_saiH}
(replacing the subindices $\cdot_k$ by $\cdot_m$ and adding 
the $\widehat{\cdot}$ sign).  This yields
\begin{equation}
\label{res_rst2b}
r_{k+m}(t) = {\widehat{\tilde{h}}_{m+1,m}}\gamma^{-1}
            [\widehat{\tilde{H}}_m^{-1}\uhat(t)]_m(I+\gamma A)\vhat_{m+1}.
\end{equation}
From~\eqref{res_rst2a} and~\eqref{res_rst2b} we see that
the residual $r_{k+m}(t)$ is, just as in~\eqref{res_rst0},
a scalar time-dependent function times a constant vector.
This shows that the derivation for $r_{k+m}(t)$ remains valid
for all Krylov-Richardson iterations (formally, we can set 
$y_k(t):=y_{k+m}(t)$ and repeat the iteration~\eqref{rich_rst}).

\section{Numerical experiments}
\label{sect:num_exp}
All our numerical experiments have been carried out with \Matlab{}
on a Linux and Mac PCs.
Unless reported otherwise, in all experiments the initial vector 
$v$ is taken to be the normalized vector with equal entries.
Except Section~\ref{sect:Cheb_exp},
in all the experiments the error reported
is the relative error norm 
with respect to a reference solution computed by the EXPOKIT
method~\cite{EXPOKIT}.
The error reported for EXPOKIT is the error 
estimate provided by this code.

\begin{figure}
\centering\epsfig{file=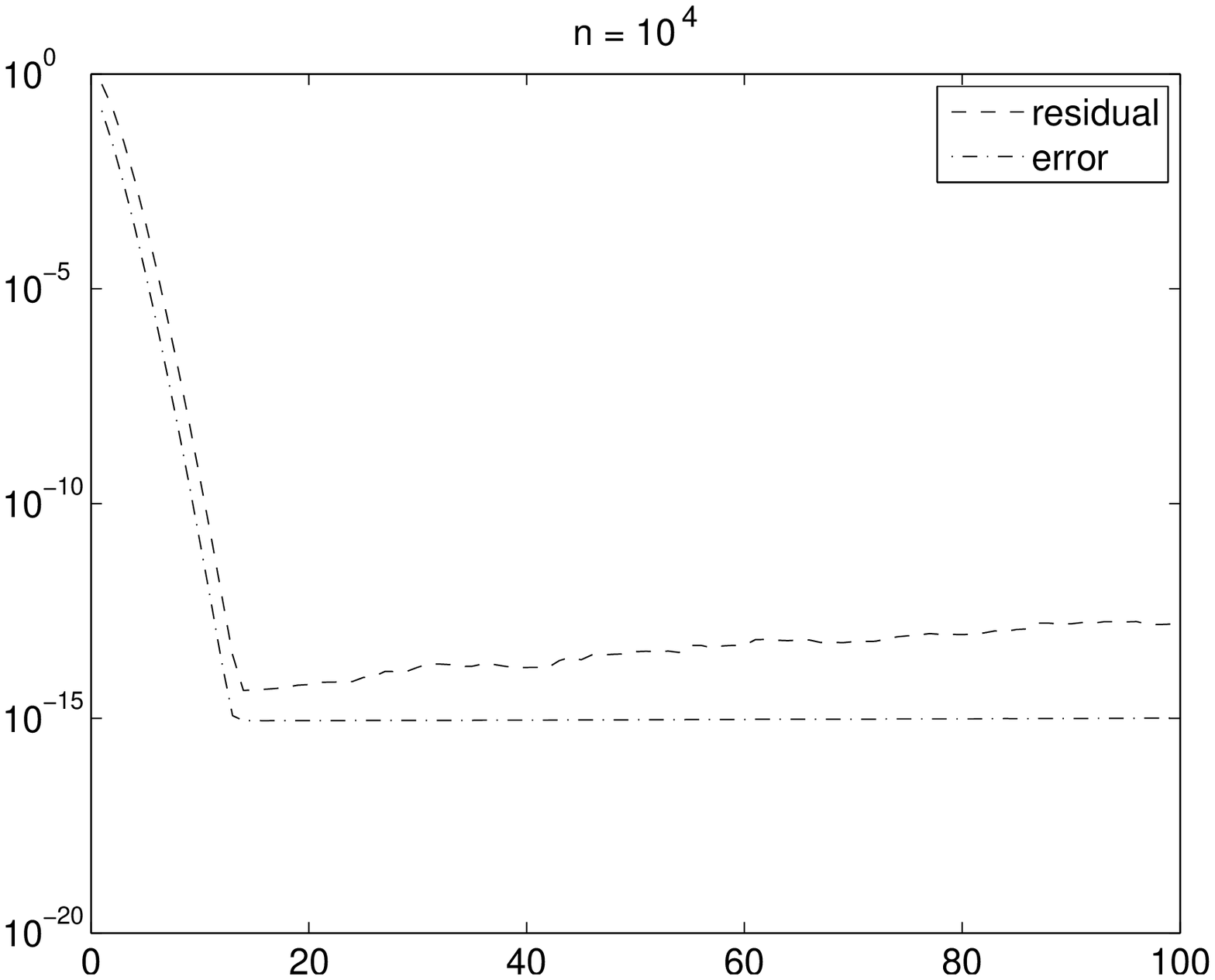,width=0.6\linewidth}
\centering\epsfig{file=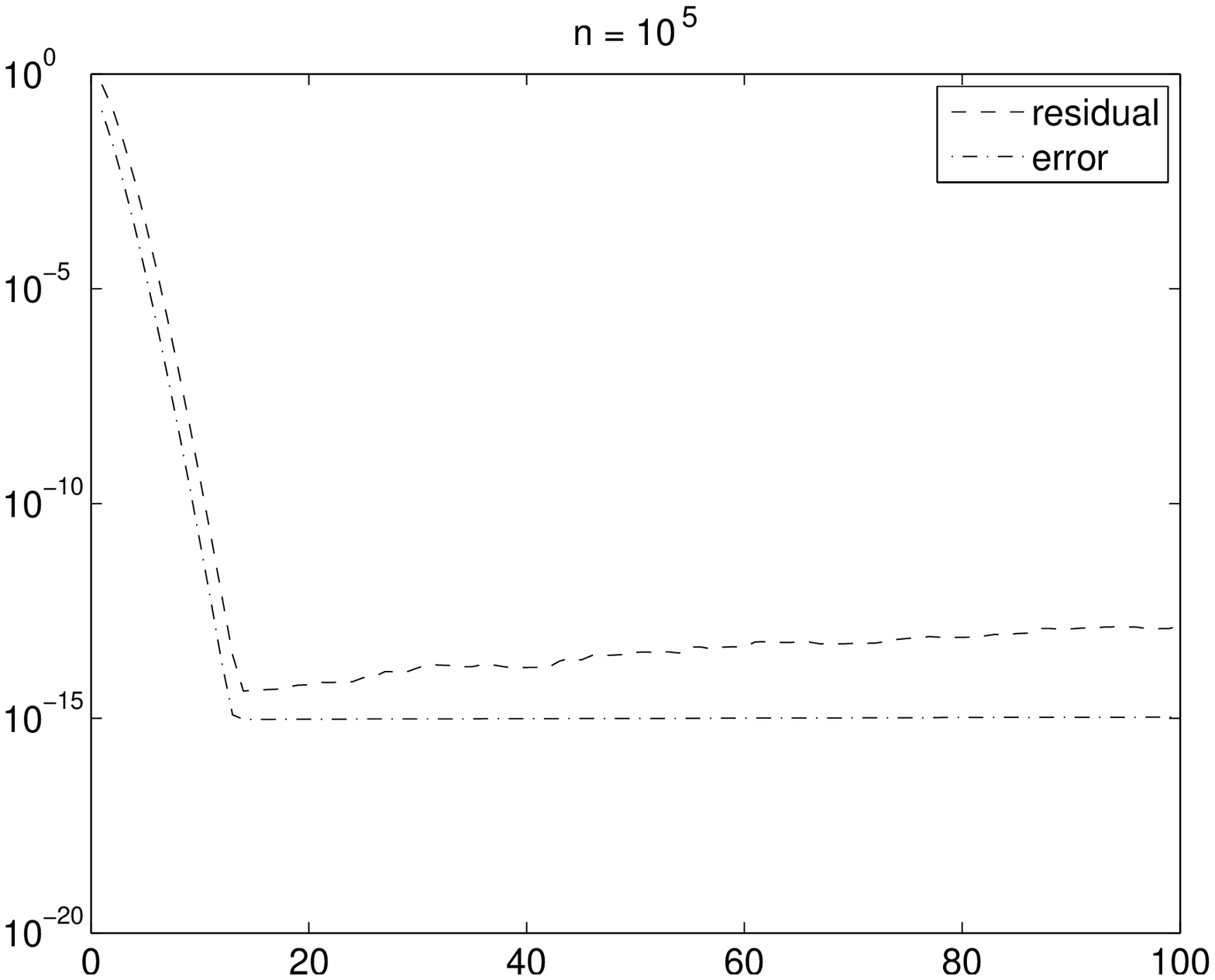,width=0.6\linewidth}
\caption{The residual and true error norms in the Chebyshev algorithm 
to compute $y_m\approx\exp(-A)v$ against iteration number $m$.  
Normal matrix $A\in\Rrnn$.
Top: $n=10^4$.  Bottom: $n=10^5$.}
\label{fig:test_chebyshev1}
\end{figure}

\subsection{Residual in Chebyshev iteration}
\label{sect:Cheb_exp}
The following tests are carried out for the Chebyshev
iterative method with incorporated residual control (see Figure~\ref{fig:chebyshev_alg}).  
We compute $\exp(-A)v$, where $v\in\Rr^n$ is a random vector 
with mean zero and standard deviation one.  In the first test, the
matrix $A\in\Rrnn$ is diagonal with diagonal entries evenly distributed 
between $-1$ and $1$.  In the second test, we fill the first superdiagonal
of $A$ with ones, so that $A$ becomes ill-conditioned.
The plots of the error and residual norms are presented in 
Figures~\ref{fig:test_chebyshev1} and~\ref{fig:test_chebyshev2}.

\begin{figure}
\centering\epsfig{file=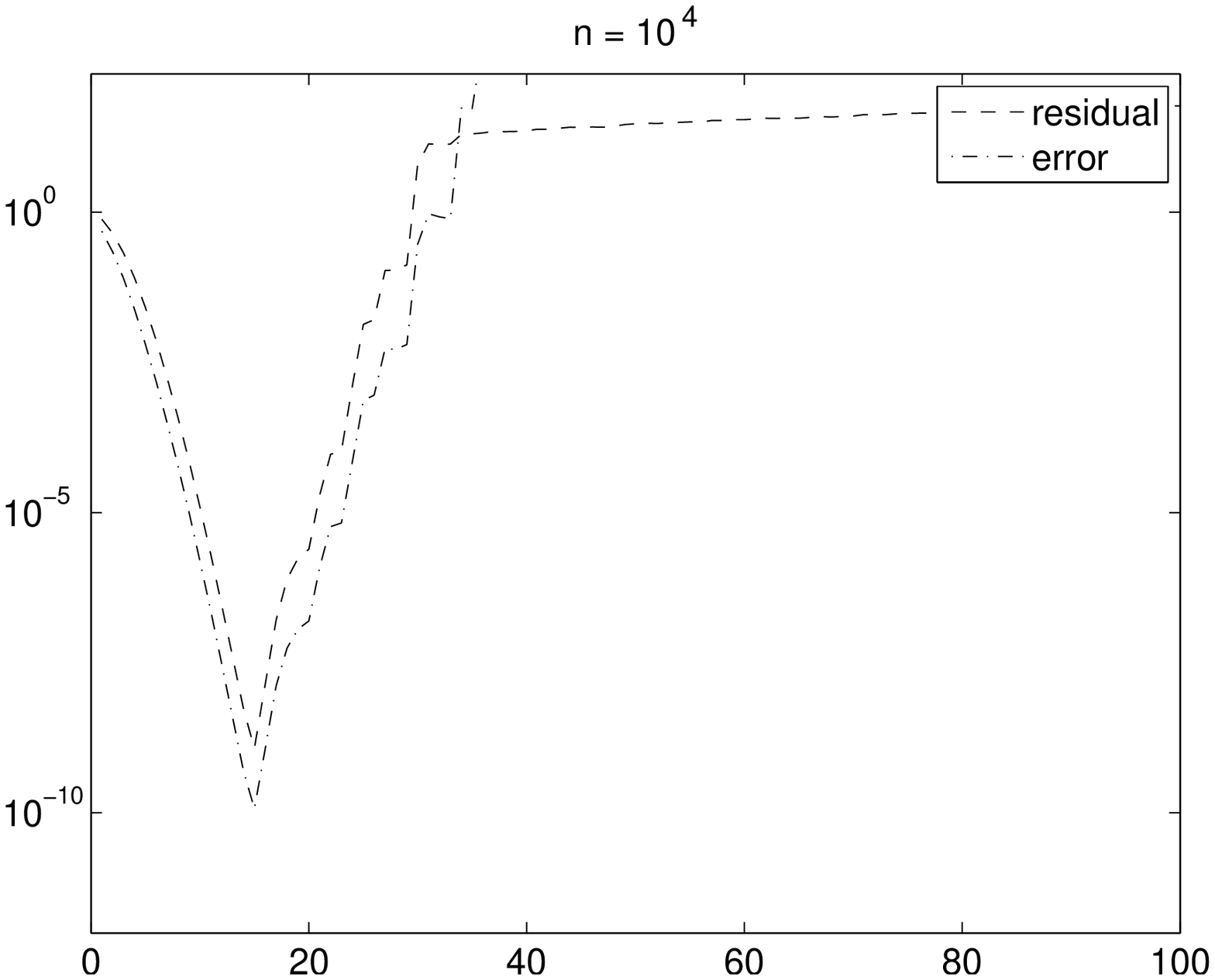,width=0.6\linewidth}
\centering\epsfig{file=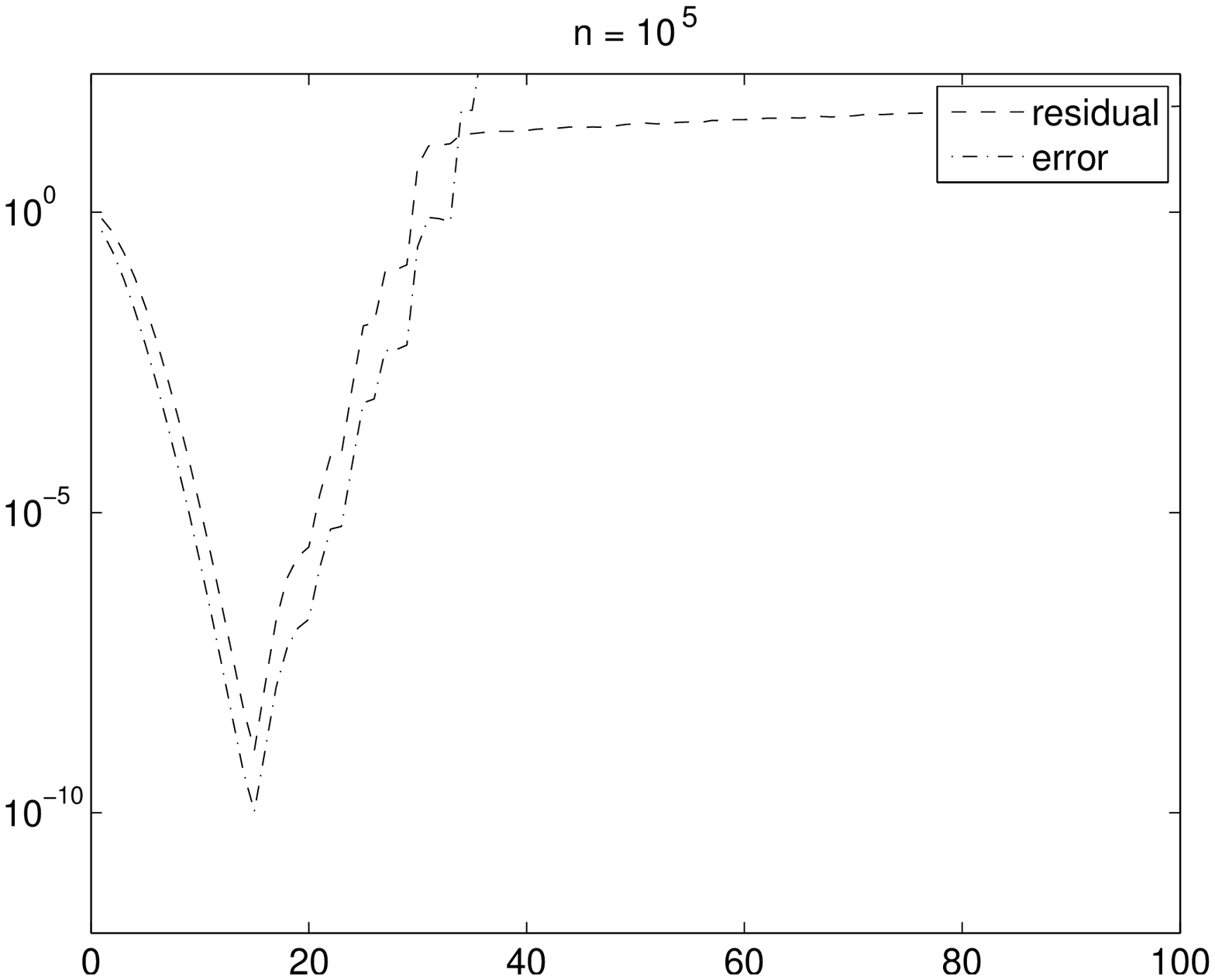,width=0.6\linewidth}
\caption{The residual and true error norms in the Chebyshev algorithm 
to compute $y_m\approx\exp(-A)v$ against iteration number $m$.  
Nonnormal matrix $A\in\Rrnn$.
Top: $n=10^4$.  Bottom: $n=10^5$.}
\label{fig:test_chebyshev2}
\end{figure}

As can be expected, for nonnormal $A$ the error accumulates during the 
iteration, so it is important to know when to stop the iteration.
Too many iterations may yield a completely wrong answer.  The residual
sharply reflects the error behavior, thus providing a reliable error
estimate. 

\subsection{A convection-diffusion problem}
\label{sect:conv_diff}
In the next several numerical experiments the matrix $A$ is taken
to be the standard five-point central difference discretization
of the following convection-diffusion operator acting on
functions defined in the domain $(x,y)\in [0,1]^2$:
\begin{gather*}
L[u]=-(D_1u_x)_x-(D_2u_y)_y + \Pe{}(v_1u_x + v_2u_y),
\\
D_1(x,y)=\begin{cases}
    10^3 \; &(x,y)\in [0.25,0.75]^2, \\
    1       &\text{otherwise},
    \end{cases}\qquad\quad D_2(x,y)=\frac12 D_1(x,y),
\\
v_1(x,y) = x+y,\qquad v_2(x,y)=x-y.
\end{gather*}
To guarantee that the convection terms yield an exactly skew-symmetric
matrix, before discretizing we rewrite the convection terms in the 
form~\cite{Krukier79}
$$
v_1u_x + v_2u_y = \frac12(v_1u_x + v_2u_y) + \frac12((v_1u)_x + (v_2u)_y).  
$$
This is possible because the velocity field $(v_1,v_2)$ is divergence
free.  The operator $L$ is set to satisfy the homogeneous Dirichlet
boundary conditions.
The discretization is done on a $102\times 102$ or $402\times 402$
uniform mesh, producing an $n\times n$ matrix $A$ of size $n=10^4$
or $n=16\times 10^4$, respectively.
The Peclet number varies from $\Pe{}=0$ (no convection,
$A=A^T$) to $\Pe{}=10^3$, which on the finer mesh means 
$\|A-A^T\|_1/\|A+A^T\|_1\approx 8\times 10^{-4}$.

\begin{table}
\centering
\caption{Performance of the exponential Richardson method 
         for the convection-diffusion test problem, 
         $\mathtt{toler}=10^{-4}$, $M=\mathrm{tridiag}(A)$.
         The CPU times are measured on a 3GHz Linux PC.
         We emphasize that the CPU time measurements are made in Matlab and thus
         are only an approximate indication of the actual performance.}\label{t:Richardson}
\begin{tabular}{cccccc}
\hline\hline
                     & flops/$n$, 
                             & matvecs      & LU           & solving      &   error\\
                     & CPU time, s
                             & $A$ / steps  & $I+\alpha M$ & $I+\alpha M$ &     \\
\hline
\multicolumn{6}{c}{$\Pe = 0$}\\\hline
EXPOKIT         & 4590, 2.6 & 918 matvecs &  ---          & ---     & 1.20e$-$11\\
exp.\ Richardson& 2192, 1.7 & 8 steps     & 24            & 176     & 2.21e$-$04\\
\hline
\multicolumn{6}{c}{$\Pe = 10$}\\\hline
EXPOKIT         & 4590, 2.6 & 918 matvecs &  ---          & ---     & 1.20e$-$11\\     
exp.\ Richardson& 2202, 1.7 & 8 steps     & 29            & 176     & 2.25e$-$04\\
\hline
\multicolumn{6}{c}{$\Pe = 100$}\\\hline
EXPOKIT         & 4590, 2.6 & 918 matvecs &  ---          & ---     & 1.20e$-$11\\
exp.\ Richardson& 2492, 1.9 & 9 steps     & 31            & 200     & 4.00e$-$04\\
\hline
\end{tabular} 
\end{table}

\begin{figure}
\centering
\epsfig{file=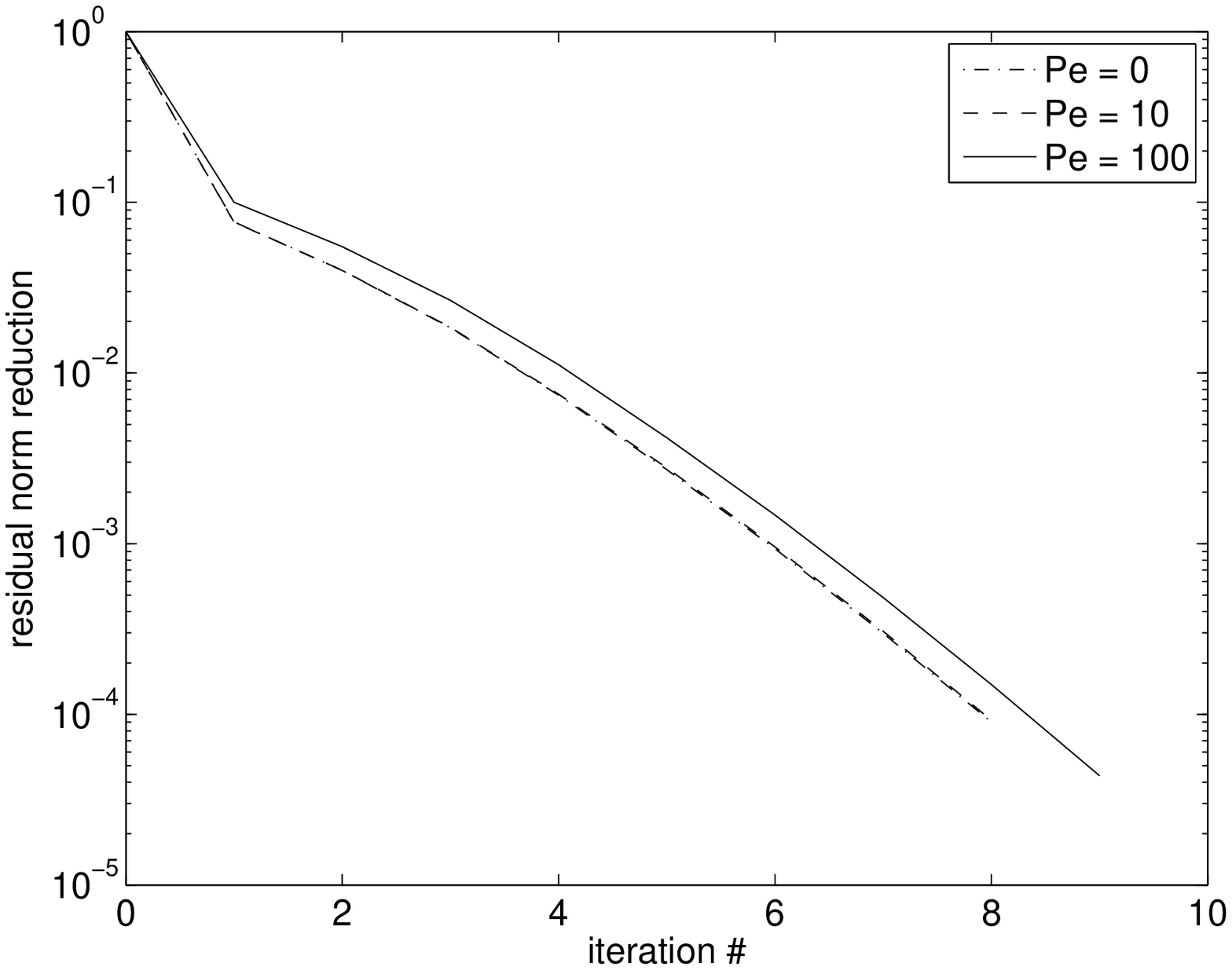,width=0.6\linewidth}  
\caption{Convergence history of the exponential Richardson iteration}
\label{f:Richardson}
\end{figure}

\subsection{Exponential Richardson iteration}
\label{sect:numerRich}
In this section we apply the exponential Richardson iteration 
\eqref{rich}, \eqref{err_rich} to
compute the vector $\exp(-A)v$ for the convection-diffusion matrices
$A$ described in Section~\ref{sect:conv_diff}. The mesh is 
taken to be $102\times 102$.
As discussed above, to be able to update the residual and solve
the IVP~\eqref{err_rich}, we need to store the values of $r_k(t)$ 
for different $t$ spanning the time interval of interest.
Too few samples may result in an accuracy loss in the interpolation
stage.  On the other hand, it can be prohibitively expensive to store 
many samples.  Therefore, in its current form, the method does not seem
practical if a high accuracy is needed.  On the other hand, it turns out 
that with relatively few samples ($\approx 20$) a moderate 
accuracy up to $10^{-5}$ can be reached.

We organize the computations in the method as follows.
The residual vector function $r_k(t)$ is stored as 20~samples.
At each iteration, the IVP~\eqref{err_rich} is solved by the 
\Matlab{} \texttt{ode15s} ODE solver, and the values of the right-hand
side function $-M\et_k(t) + r_k(t)$ are interpolated using the 
stored samples.  The \texttt{ode15s} solver is run with tolerances
determined by the final required accuracy and produces
the solution $\et_k(t)$ in the form of its twenty samples.
Then, the solution and residual are updated according to~\eqref{rich} 
and~\eqref{res_rich1} respectively.

We have chosen $M$ to be the tridiagonal part $\mathrm{tridiag}(A)$
of the matrix $A$.  Table~\ref{t:Richardson} and 
Figure~\ref{f:Richardson} contains results
of the test runs.  Except the Richardson method, as a reference
we use the EXPOKIT code~\cite{EXPOKIT} with the maximal Krylov 
dimension 100.  
Note that EXPOKIT provides a much accurate solution than requested
by the tolerance $\mathtt{toler}=10^{-4}$.

It is rather difficult
to compare the total computational work of the EXPOKIT and Richardson
methods exactly.  We restrict ourselves to the matrix-vector part of the work.
In the Richardson method this work consists of the matrix-vector
multiplication (matvec) with $M-A$ in~\eqref{res_rich1} and 
the work done by the \texttt{ode15s} solver.  The matvec with 
bidiagonal $M-A$ 
costs about $3n$ flops times 20~samples, in total $60n$ 
flops\footnote{We use definition of flop from 
\cite[Section~1.2.4]{GolVanL}.}.
The linear algebra work in \texttt{ode15s} is essentially
tridiagonal matvecs, LU~factorizations and back/forward substitutions
with (possibly shifted and scaled) $M$.
According to
\cite[Section~4.3.1]{GolVanL}, tridiagonal LU~factorization,
back- and forward substitution require about $2n$ flops each.
A matvec with tridiagonal $M$ is $5n$~flops.  Thus,
in total exponential Richardson costs $60n$~flops times the number of iterations
plus $2n$~flops times the number of LU~factorizations and back/forward
substitutions plus $5n$~flops times the total number of ODE solver steps.
The matvec work in EXPOKIT consists of matvecs with pentadiagonal $A$,
which is about $9n$ flops.  

From Table~\ref{t:Richardson} we see that exponential Richardson
is approximately twice as cheap as EXPOKIT.  
As expected from the convergence estimates, exponential
Ri\-chard\-son converges much faster than the conventional Richardson 
iteration for solving a linear system $Ax=v$ would do.  For these 
$A$ and $v$, 
8--9 iterations of the conventional Richardson would only give a residual 
reduction by a factor of $\approx 0.99$.

\begin{figure}
\epsfig{file=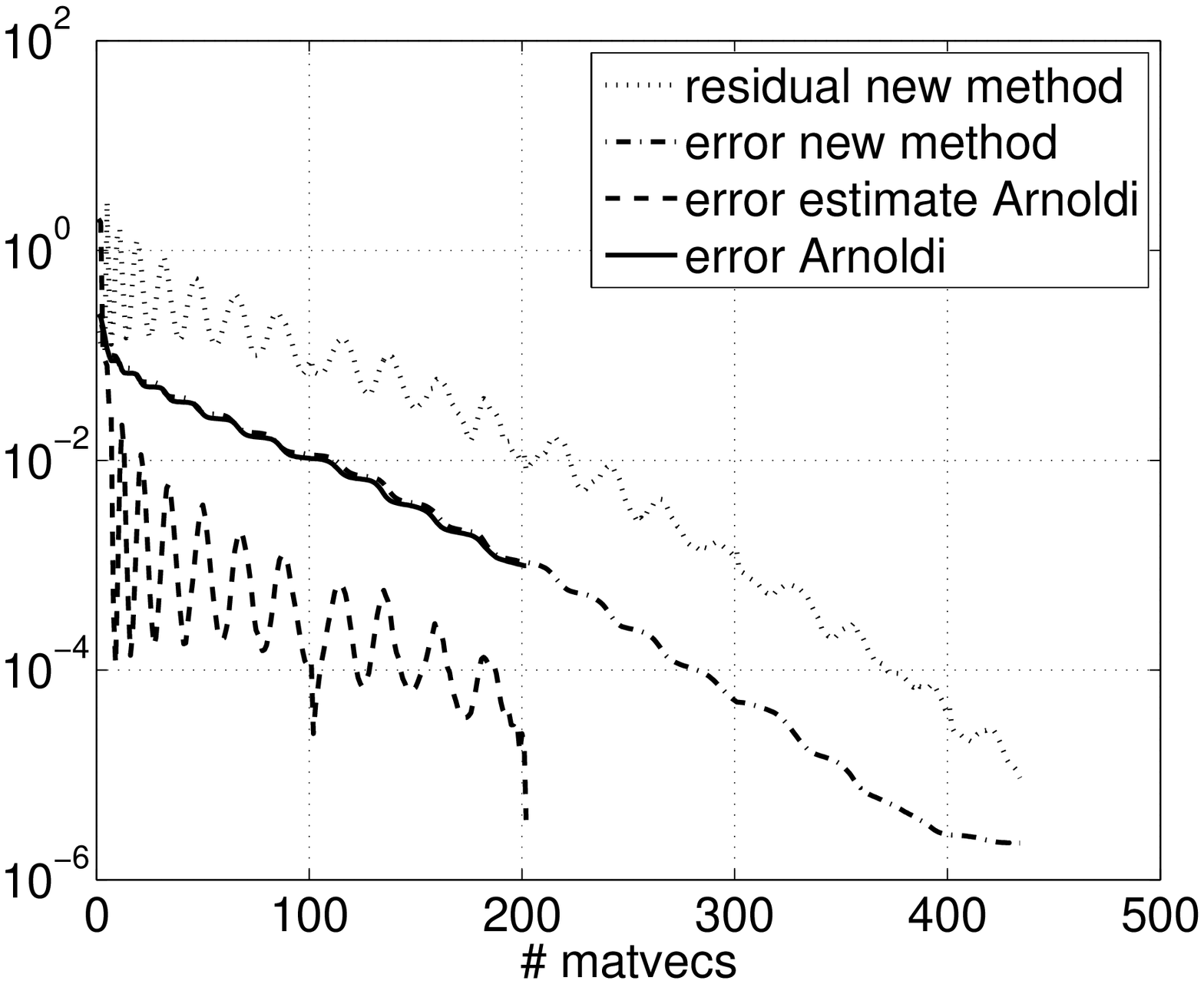,width=0.49\linewidth}%
\epsfig{file= 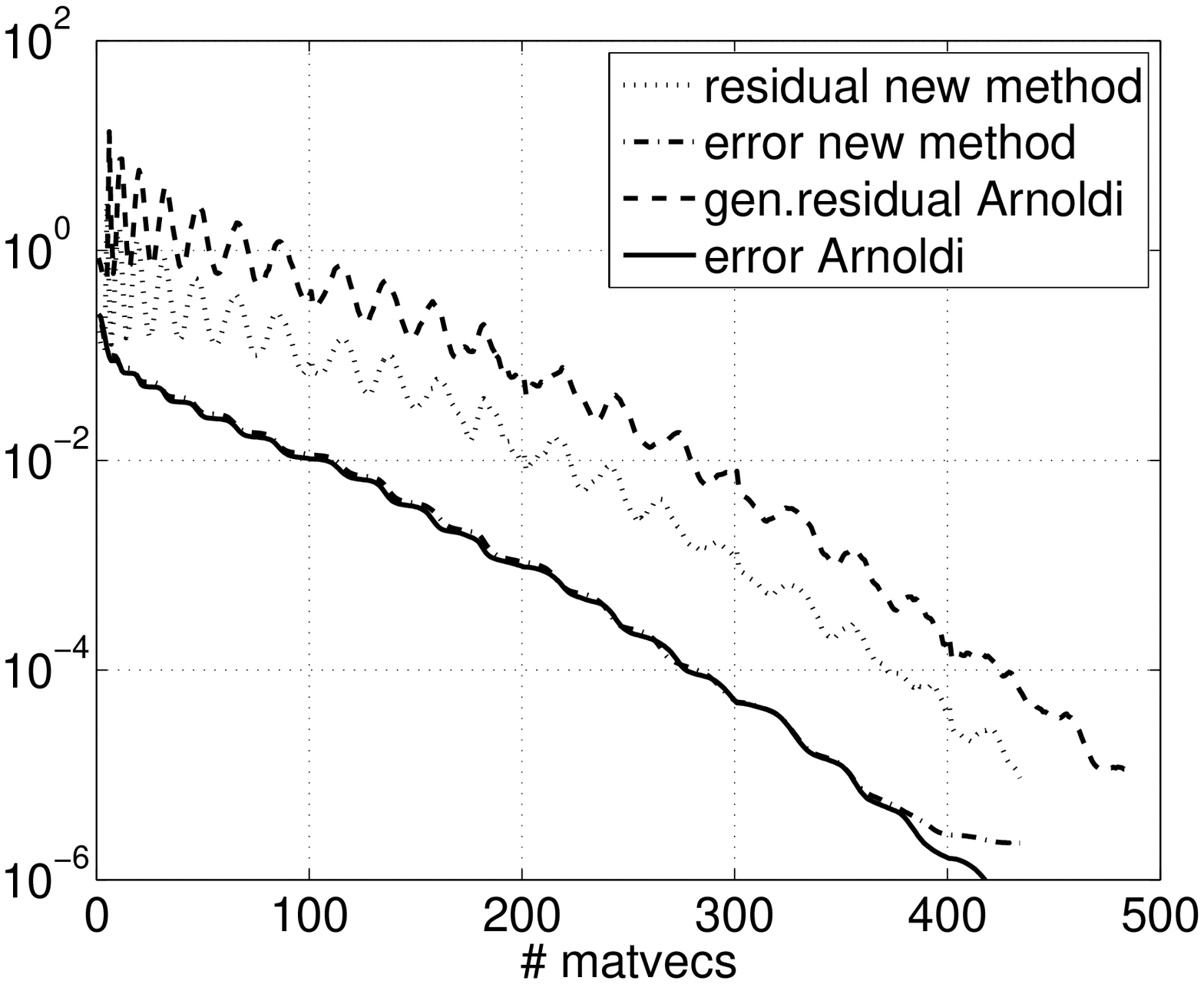,width=0.49\linewidth}
\caption{Convergence of the conventional Arnoldi method with two existing 
stopping criteria and Krylov-Richardson with the residual-based
stopping criterion for tolerance $\mathtt{toler}=10^{-5}$.
Left: stagnation-based criterion, Arnoldi stops too early
(201~matvecs, 2.6~s CPU time, error 1.0e$-$03).
Right: generalized residual criterion, Arnoldi stops too 
late 
(487~matvecs, 139~s CPU time, error 4.9e$-$08).  
Parameters of the Krylov-Richardson run for both plots: 
434~matvecs, 11~s CPU time, error 2.2e$-$06).
The CPU measurements (on a 3GHz Linux PC) are made in Matlab
and thus are only an indication of the actual performance.}
\label{fig:stop}  
\end{figure}

\subsection{Experiments with Krylov-Richardson iteration}
\label{sect:Krylov_tests}
In this section we present some numerical experiments with
the Krylov-Richardson method presented in Section~\ref{sect:Krylov-Richardson}.
We now briefly describe the other methods to which Krylov-Richardson
is compared.

\begin{figure}
\epsfig{file=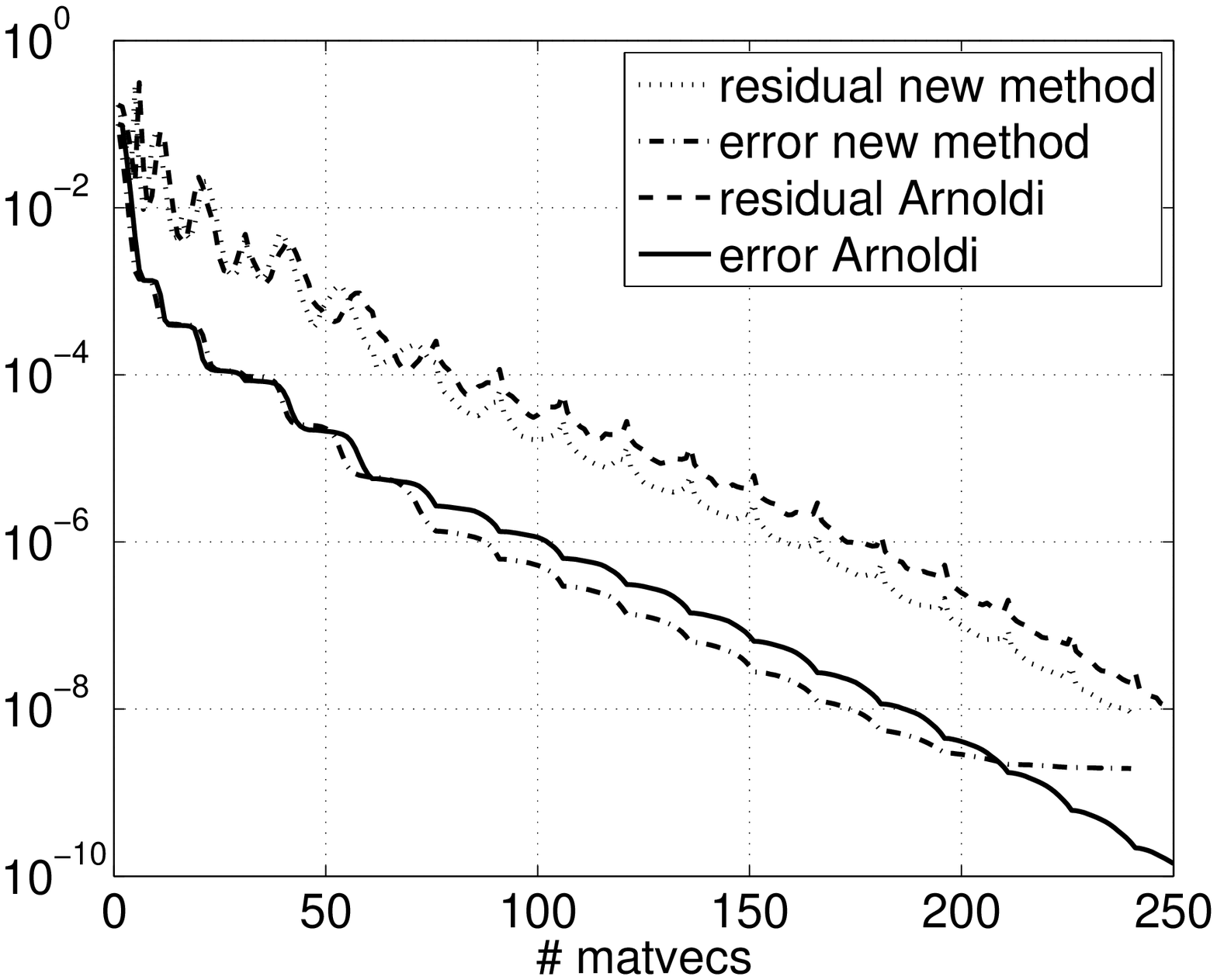,width=0.49\linewidth}%
\epsfig{file=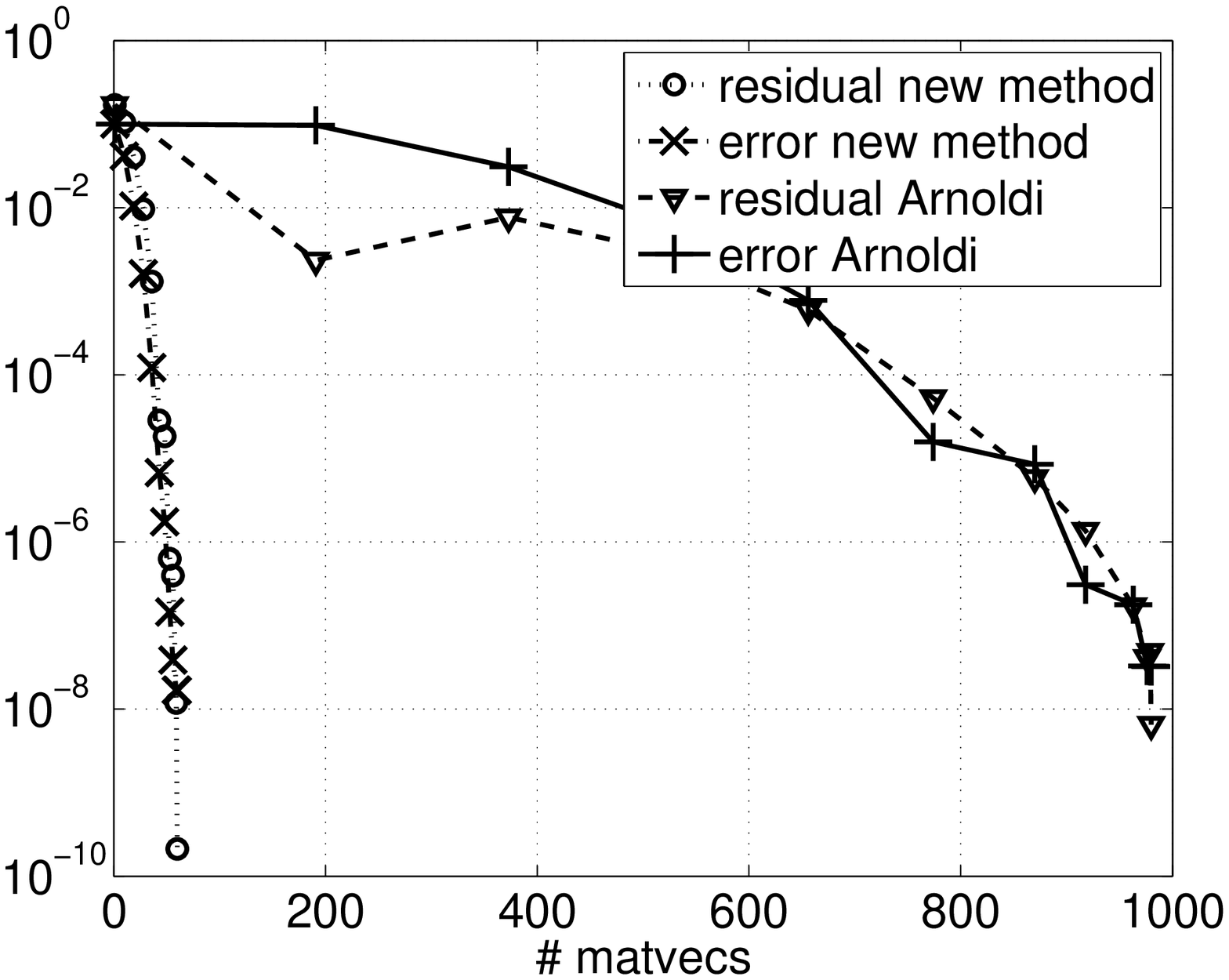,width=0.49\linewidth}
\caption{Convergence plots of the Arnoldi/Lanczos and the new 
Krylov-Richardson methods, mesh $102\times 102$, $\Pe{}=100$.
Left:  restart every 15~steps,
right: SaI strategy with GMRES.
The peaks in the residual plots on the left correspond to the restarts.}
\label{f:main_results}
\end{figure}

Together with the classical
Arnoldi/Lanczos method \cite{GallSaad92,Saad92,DruskinKnizh95,%
HochLub97}, we have tested the SaI method
of Van~den~Eshof and Hochbruck~\cite{EshofHochbruck06}.  We have implemented the method
exactly as described in their paper, with a single modification.  
In particular, as advised by the authors,
in all the tests the shift parameter $\gamma$ is set to $0.1t_{\text{end}}$
and the relaxed stopping criterion strategy for the inner iterative SaI
solvers is employed.  The only thing we have changed is the stopping
criterion of the outer Krylov process.  To be able to exactly compare
the computational work, we can switch from the stopping criterion of Van~den~Eshof
and Hochbruck (based on iteration stagnation) to the residual
stopping criterion (see Lemma~\ref{lemma2}).  
Note that the relaxed strategy for the inner SaI solver is
then also based on the residual norm and not on the error estimate.

Since the Krylov-Richardson method is essentially a restarting technique,
it has to be compared with other existing restarting techniques.
Note that a number of restarting strategies have recently
been developed~%
\cite{TalEzer2007,Afanasjew_ea08,PhD_Guettel,EiermannErnstGuttel09,PhD_Niehoff}.  
We have used the restarting method described
in~\cite{PhD_Niehoff}.  This choice is motivated by the fact that
the method from~\cite{PhD_Niehoff} turns out to be algorithmically very close
to our Krylov-Richardson method.  In fact, the only essential 
difference is handling of the projected problem.  In the method~\cite{PhD_Niehoff}
the projected matrix $H_k$ built up at every restart is appended to
a larger matrix $\widetilde{H}_{*+k}$.  There, the projected matrices from
each restart are accumulated.  Thus, if 10 restarts of 20 steps are
done, we have to compute the matrix exponential of a $200\times 200$
matrix.  In our method, the projected matrices are not accumulated,
so at every restart we deal with a $20\times 20$ matrix.  The price
to pay, however, is the solution of the small IVP~\eqref{ivp_uhat}.

In our implementation, at each Krylov-Richardson iteration the 
IVP~\eqref{ivp_uhat} is solved by the \texttt{ode15s}
ODE solver from \Matlab{}.  To save computational work, it is
essential that the solver be called most of the time with a relaxed 
tolerance (in our code we set the tolerance to 1\% of the
current residual norm).  This is sufficient to estimate
the residual accurately.  Only when the actual solution update takes place
(see formula~\eqref{rich_rst}) do we solve the projected IVP to a 
full accuracy.

Since the residual time dependence in Krylov-Richardson is given
by a scalar function, little storage is needed for the look-up
table.  Based on the required accuracy, the \texttt{ode15s} solver 
automatically determines how many samples need to be stored
(in our experiments this usually did not exceed 300).
This happens at the end of each restart or when
the stopping criterion is satisfied.
Further savings in computational work can be achieved by a
polynomial
fitting: at each restart the newly computed values of the
function $\psi_k$ (see~\eqref{res_rst0}) are approximated by a 
best-fit
polynomial of a moderate degree (in all experiments the degree was
set to~6).  If the fitting error is too large (this 
depends on the required tolerance), the algorithm proceeds
as before.  Otherwise, the $\psi_k$ function
is replaced by its best-fit polynomial.  This allows
a faster solution of the projected IVP~\eqref{ivp_uhat}
through an explicit formula containing the functions
$$
\varphi_k(x)=\frac{\varphi_{k-1}(x)-\varphi_{k-1}(0)}{x},\quad
k\geqs 1,\quad \varphi_0=e^x.
$$

We now present an experiment showing the importance of a proper
stopping criterion.  We compute $\exp(-5A)v$ for $A$ being
the convection-diffusion operator discretized 
on a uniform mesh $102\times 102$ with $\Pe{}=100$.  
The tolerance is set to $\mathtt{toler}=10^{-5}$.
We let the usual Arnoldi method,
restarted every 100~steps, run with the stagnation-based
stopping criterion of~\cite{EshofHochbruck06} and with
the stopping criterion of~\cite{HochLubSel97} based on the generalized 
residual~\eqref{gen_resid}.  We emphasize that the stagnation-based
stopping criterion of~\cite{EshofHochbruck06} is proposed for the
Arnoldi method with SaI strategy and it works, in our limited experience,
very well as soon as SaI is employed.  However, the stagnation-based 
stopping criteria are used in Krylov methods not only
with SaI (see e.g.~\cite{Gautschi2006}) and it is instructive to see possible
implications of it.  Together with Arnoldi, the Krylov-Richardson method
is run with the residual-based stopping criterion.  The 
convergence plots are shown in Figure~\ref{fig:stop}.
As we see, both existing stopping criteria turn out to be 
far from optimal in this test.
With the residual-based stopping criterion, the Arnoldi
method required 438~matvecs and 78~s CPU time to obtain
an adequate accuracy of 4.5e$-$7.

\begin{table}
\caption{Results of the test runs of the Krylov-Richardson 
         and Arnoldi with the residual-based stopping criterion.  
         The CPU times are measured on a 2GHz Mac PC
         (mesh $102\times 102$) and on a 3GHz Linux PC (mesh $402\times 402$).
         We emphasize that the CPU time measurements are made in Matlab and thus
         are only an approximate indication of the actual performance.}
\label{t:main_results}
\begin{center}
\begin{tabular}{ccccc}
\hline\hline
           & restart~/~SaI & total matvecs & CPU      & error\\
           &               & or LU actions & time     &      \\
\hline
\multicolumn{5}{c}{mesh $102\times 102$, $\Pe = 100$}\\\hline
EXPOKIT    & restart 15    & 1343              &  2.2  &  3.61e$-$09\\
Arnoldi    & restart 15    & 250               & 26.4  &  1.45e$-$10\\
new method & restart 15    & 240               &  6.7  &  1.94e$-$09\\
EXPOKIT    & restart 100   & 1020              & 7.6     &1.33e$-$11\\  
Arnoldi    & restart 100   & 167               & 7.9     &1.21e$-$10\\
new method & restart 100   & 168               & 11.8    &1.14e$-$10\\   
Arnoldi    & SaI/GMRES$^a$ & 980 (11 steps)    & 17.8 & 3.29e$-$08\\
new method & SaI/GMRES$^a$ & 60  (10 steps)    & 1.7  & 1.67e$-$08\\
Arnoldi    & SaI/sparse LU & ``11'' (10 steps) & 1.7  & 3.62e$-$09\\
new method & SaI/sparse LU & ``11'' (10 steps) & 1.8  & 1.61e$-$10\\
\hline
\multicolumn{5}{c}{mesh $402\times 402$, $\Pe = 1000$}\\\hline
EXPOKIT    & restart 15     & 1445          &  21     &  4.36e$-$09\\
Arnoldi    & restart 15     & 244           &  11     &  1.13e$-$10\\
new method & restart 15     & 254           &  15     &  2.62e$-$09\\
EXPOKIT    & restart 100    & 1020          &  69     & 1.33e$-$11\\  
Arnoldi    & restart 100    & 202           &  34     & 1.06e$-$10\\
new method & restart 100    & 200           &  35     & 3.62e$-$10\\   
Arnoldi    & SaI/GMRES$^a$  &1147 (15 steps)    &  80  & 5.68e$-$08\\
new method & SaI/GMRES$^a$  & 97  (12 steps)    & 6.2  & 1.28e$-$08\\
Arnoldi    & SaI/sparse LU  & ``12'' (11 steps) &  46  & 3.06e$-$08\\
new method & SaI/sparse LU  & ``13'' (12 steps) &  50  & 2.07e$-$10\\
\hline
\multicolumn{3}{l}{$^a$~GMRES(100) with SSOR preconditioner}
\\\hline
\end{tabular}
\end{center}    
\end{table}

To facilitate a fare comparison between the conventional Arnoldi and
the Krylov-Richardson methods, in all the other tests we use
the residual-based stopping criterion for both methods.
Table~\ref{t:main_results} and Figures~\ref{f:main_results} contain the results of the 
test runs to compute $\exp(-A)v$ for tolerance $\mathtt{toler}=10^{-8}$.
We show the results on two meshes for two different Peclet numbers only, 
the results for other Peclet numbers are quite similar.

The first observation we make is that the CPU times measured in Matlab
seem to favor the EXPOKIT code, disregarding  the actual matvec values.
We emphasize that when the SaI strategy is not used, the main 
computational cost in all the three methods, EXPOKIT, Arnoldi and Krylov-Richardson,
are $k$ steps of the Arnoldi/Lanczos process.  The differences among the 
three methods correspond to the rest of the computational
work, which is $\mathcal{O}(k^3)$, if at least if not too many restarts
are made.

The second observation is that the convergence of the Krylov-Richardson
iteration is essentially the same as of the classical Arnoldi/Lanczos
method.  This is not influenced by the restart value or by the SaI
strategy.  Theoretically, this is to be expected: the former
method applies Krylov for the $\varphi$ function, the latter for
the exponential; for both functions similar convergence estimates
hold, though they are slightly more favorable for the $\varphi$ 
function~\cite{HochLub97}.

When no SaI strategy is applied, the gain we have with Krylov-Richardson
is two-fold.  First, a projected problem of much smaller size has to 
be solved.  This is reflected by the difference in the CPU times of 
Arnoldi and Krylov-Richardson with restart 15 in lines 2 and 3 of the Table:
26.4~s and 6.7~s.  Of course, this effect can be less pronounced
for larger problems or on faster computers---see the corresponding
lines for Arnoldi and Krylov-Richardson with restart 15 on a finer mesh.
Second, we have some freedom in choosing the initial
vector (in standard Arnoldi/Lanczos we must always start
with $v$).  This freedom is not complete because the residual of the 
initial guess has to have scalar dependence on time.  Several 
variants for choosing the initial vector exist, and we will
explore these possibilities in the future.

A significant reduction in total computational work can be achieved
when Kry\-lov-Richard\-son is combined with the SaI strategy.
The gain is then due to the reduction in the number of the inner
iterations (the number of outer iterative steps is approximately 
the same).  In our limited experience, this is not always the case 
but typically takes place
when, for instance, $v$ and $A$ represent discretizations of
a smooth function and a smooth partial differential operator,
respectively.
Currently, we do not completely understand
this behavior.  Apparently, the Krylov subspace vectors
built in the Krylov-Richardson method constitute more favorable
right-hand sides for the inner SaI solvers to converge.  It is
rather difficult to analyze this phenomenon, but we will 
try to do this in the near future.

\subsection{Initial vector and Krylov subspace convergence}
It is instructive to observe dependence of the Krylov subspace methods
on the initial vector $v$.  In particular, if~\eqref{ivp} stems from an
initial-boundary-value problem (IBVP) and $A$ is a discretized
partial differential operator, a faster convergence may take place
for $v$ satisfying boundary conditions of the problem.
Note that for the convection-diffusion 
test problem from the previous
section this effect is not pronounced ($v$ did not satisfy
boundary conditions), probably due to the
jump in the diffusion coefficients.  We therefore demonstrate this
effect on a simple IBVP
\begin{equation}
\label{ibvp}
u_t = \Delta u, \qquad u(x,y,z,0)=u_0(x,y,z),
\end{equation}
posed for $(x,y,z)\in [0,1]^3$ for unknown function $u(x,y,z,t)$ obeying
periodic boundary conditions.    
We use a fourth-order finite volume 
discretization in space from~\cite{VerstappenVeldman2003} on a regular 
mesh $40\times 40\times 40$ and arrive at IVP~\eqref{ivp}
which we solve for $t=1000$ by computing $\exp(-tA)v$.
In Figure~\ref{fig:v} convergence of the Krylov-Richardson and
Arnoldi/Lanczos methods is illustrated for the starting vector
$v$ corresponding to
$$
u_0(x,y,z) = \sin(2\pi x)\sin(2\pi y)\sin(2\pi z) + x(a-x)y(a-y)z(a-z),
$$
with $a=2$ or $a=1$.  In both cases the restart value is set to~100.
The second choice $a=1$ (right plot)
agrees with boundary conditions in the sense that $u_0$ can
be periodically extended and leads to a faster convergence.
The same effect is observed for the Krylov-Richardson and
Arnoldi/Lanczos methods with the SaI strategy, with a reduction
in the number of steps from 12 to 8 or~9.
Remarkably, EXPOKIT(100) converges for both choices of $v$ within
the same number of steps,~306.  Apparently, this is because EXPOKIT splits
the given time interval $[0,t_{\text{end}}]$, building for each
subinterval a new Krylov subspace.

\begin{figure}
\epsfig{file=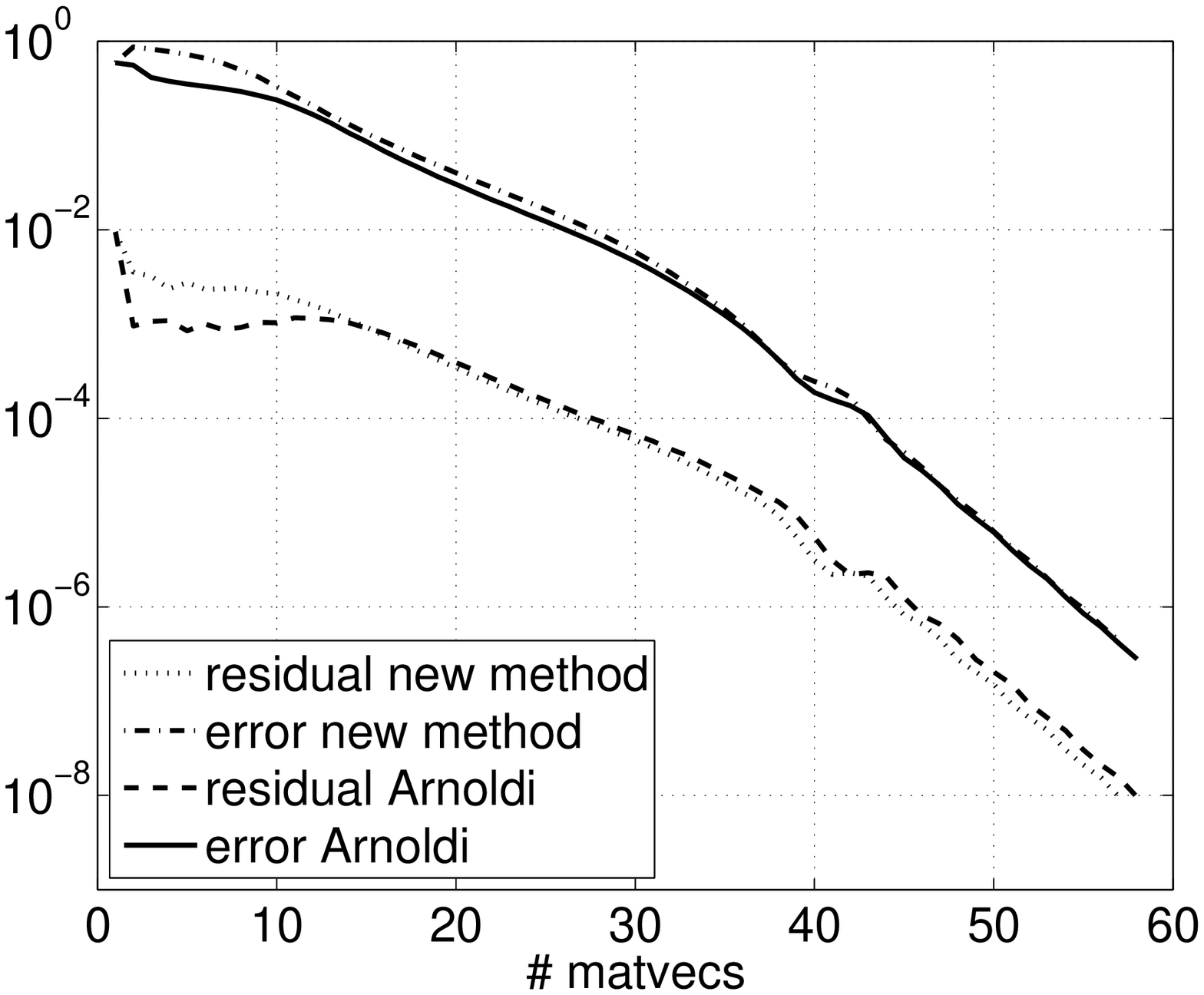,width=0.49\linewidth}%
\epsfig{file=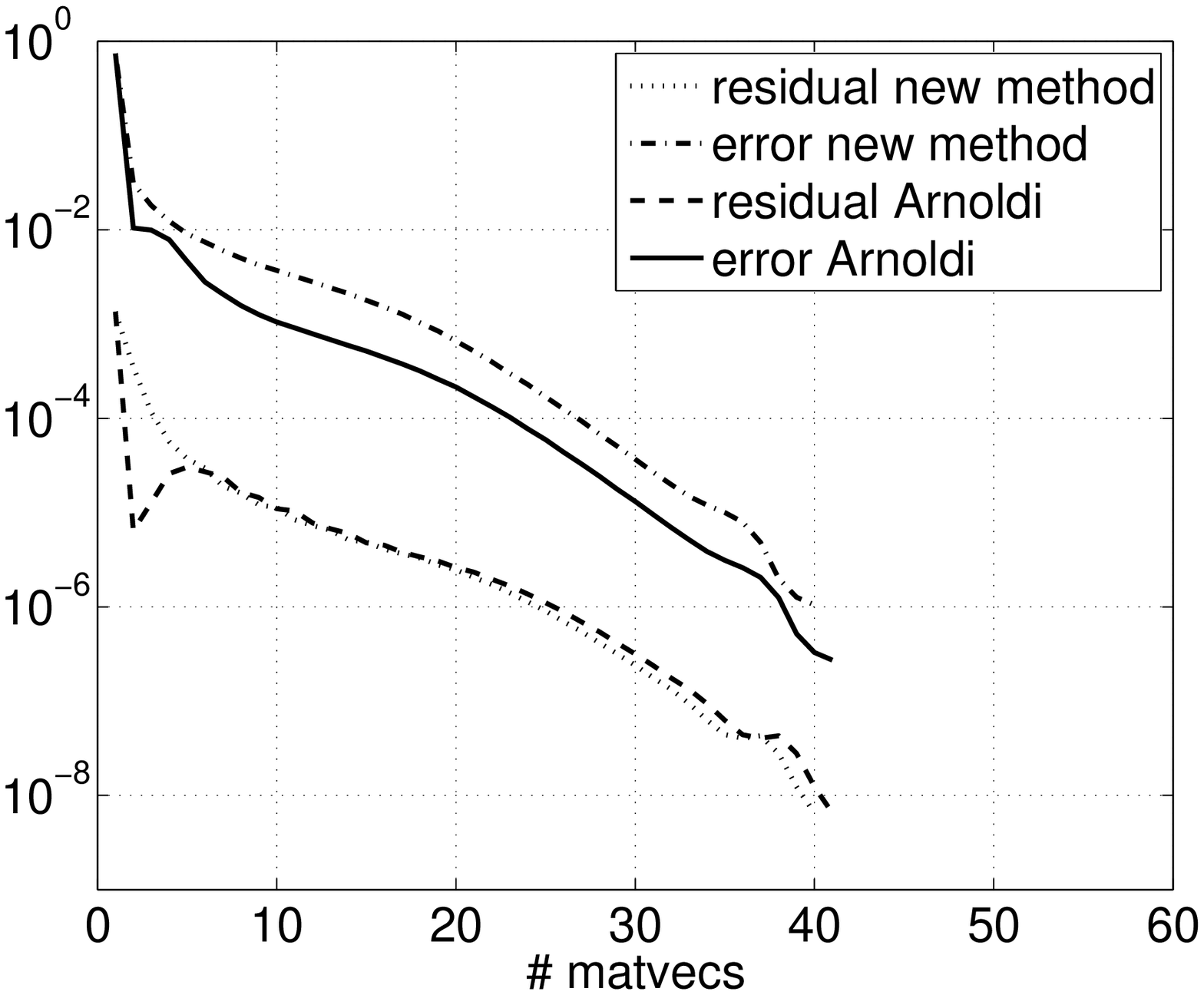,width=0.49\linewidth}
\caption{Convergence plots of the Arnoldi/Lanczos and the new 
Krylov-Richardson methods for the fourth-order finite volume discretization
of the three-dimensional Laplacian with periodic boundary conditions.
Left: the starting vector $v$ disobeys the boundary conditions,
right: $v$ is consistent with the boundary conditions.}
\label{fig:v}
\end{figure}

\section{Concluding remarks and an outlook to further research}
The proposed residual notion appears to provide a reliable stopping
criterion in the iterative methods for computing the matrix exponential.
This is confirmed by the numerical tests and analysis.
Furthermore, the residual concept seems to set up a whole framework for 
a new class of methods for evaluating the matrix exponential.
Some basic methods of this class are proposed in this paper.
Many new research questions arise.  One of them is 
a comprehensive convergence analysis of the new exponential Richardson
and Krylov-Richardson methods.  Another interesting research direction
is development of other residual-based iterative methods.  
In particular, one may ask whether the exponential Richardson
method~\eqref{rich}--\eqref{err_rich} can not be used as a preconditioner 
for the Krylov-Richardson method~\eqref{rich_rst}.  We plan to address this
question in future.

Finally, 
an interesting question is whether the proposed residual notion can be extended
to other matrix functions.  This is possible once a residual equation
can be identified, i.e.\ an equation such that the matrix function
satisfies this equation (see Table~\ref{t1}).  For example, if we are
interested in computing the vector $u=\cos(A)v$,
for given $A\in\Rrnn$ and $v\in\Rr^n$, then we may consider a vector
function $u(t)=\cos(tA)v$, which is a solution of the IVP
$$
u''(t)=-A^2u, \quad u(0)=v, \quad u'(0)=0.
$$
Thus, for an approximate solution $u_k(t)\approx u(t)$ satisfying
the initial conditions, the residual can be introduced as 
$$
r_k(t)\equiv -A^2u_k(t)-u_k''(t).
$$

\section*{Acknowledgments}
The author would like to thank anonymous referees
and a number of colleagues, in particular,
Michael Saunders, Jan Verwer and Julien Langou
for valuable comments and
Marlis Hoch\-bruck for explaining the ideas from~\cite{PhD_Niehoff}.

\bibliography{my_bib,matfun}
\bibliographystyle{abbrv}

\end{document}